\documentclass[11pt]{article}

\usepackage{natbib}
\usepackage[top=1.5in,left=1.25in,right=1.25in,bottom=1in]{geometry}

\usepackage[utf8]{inputenc} 
\usepackage[T1]{fontenc}    
\usepackage{hyperref}       
\usepackage{url}            
\usepackage{booktabs}       
\usepackage{amsfonts}       
\usepackage{nicefrac}       
\usepackage{microtype}      

\usepackage{paralist, amsmath, amsthm, amssymb, color, graphicx, hyperref}
\DeclareGraphicsExtensions{.jpg,.pdf,.eps}
\usepackage{amssymb}
\usepackage{setspace}
\usepackage{mathtools}
\usepackage{color}
\usepackage{algorithm,algorithmic}
\usepackage{float}
\usepackage{caption}

\usepackage{paralist}

\usepackage{stmaryrd}

\usepackage{xfrac}

\usepackage{mathrsfs}
\usepackage{amscd}
\usepackage{dsfont}
\usepackage{tikz}

\usepackage{enumerate}

\theoremstyle{plain}
\newtheorem{thm}{Theorem}[section] 
\newtheorem{lem}[thm]{Lemma}
\newtheorem{prop}[thm]{Proposition}
\newtheorem{cor}[thm]{Corollary}

\theoremstyle{definition}
\newtheorem{defn}[thm]{Definition} 

\newtheorem{remark}[thm]{Remark}

\usepackage{listings}
\usepackage{color} 
\definecolor{mygreen}{RGB}{28,172,0} 
\definecolor{mylilas}{RGB}{170,55,241}
\definecolor{mygray}{gray}{0.95}

\newcommand{\mscr}[1]{\mathscr{#1}}

\newcommand{\twid}[1]{\widetilde{#1}}

\newcommand{\ZZ}{\mathbb{Z}}
\newcommand{\RR}{\mathbb{R}}
\newcommand{\NN}{\mathbb{N}}

\newcommand{\EE}{\mathbb{E}}

\newcommand{\al}{\alpha}

\newcommand{\ga}{\gamma}
\newcommand{\de}{\delta}
\newcommand{\ep}{\epsilon}

\newcommand{\ta}{\theta}
\newcommand{\Ta}{\Theta}

\newcommand{\sa}{\sigma}

\usepackage{mathtools}

\renewcommand{\l}{\left}
\renewcommand{\r}{\right}

\newcommand{\defeq}{\vcentcolon=}

 \newcommand{\iid}{\overset{\text{iid}}{\sim}}

\DeclareMathOperator{\var}{Var}

\DeclareMathOperator{\floor}{floor}
\DeclareMathOperator{\ceil}{ceil}

\newcommand{\ul}{\underline}
\newcommand{\ds}{\displaystyle}

\title{Differentially Private Inference for Binomial Data}

\author{Jordan Awan \& Aleksandra Slavkovi\'c}
\author{
  Jordan Awan\\
  Department of Statistics\\
  Penn State University\\
  University Park, PA 16802 \\
  \texttt{awan@psu.edu} 
  \and
  Aleksandra Slavkovi\'c\\
  Department of Statistics\\
  Penn State University\\
  University Park, PA 16802 \\
  \texttt{sesa@psu.edu} \\ 
}
\date{}

\begin{document}

\maketitle

\begin{abstract}
  We derive uniformly most powerful (UMP)  tests for simple and one-sided hypotheses for a population proportion within the framework of Differential Privacy (DP), optimizing finite sample performance. We show that in general,  DP hypothesis tests can be written in terms of linear constraints, and for exchangeable data can always be expressed as a function of the empirical distribution. Using this structure, we prove a `Neyman-Pearson lemma' for binomial data under DP, where the DP-UMP only depends on the sample sum. Our tests can also be stated as a post-processing of a random variable, whose distribution we coin ``Truncated-Uniform-Laplace'' (Tulap), a generalization of the Staircase and discrete Laplace distributions.   Furthermore, we obtain exact $p$-values, which are easily computed in terms of the Tulap random variable.

  Using the above techniques, we show that our tests can be applied to give uniformly most accurate one-sided confidence intervals and optimal confidence distributions. We also derive uniformly most powerful unbiased (UMPU) two-sided tests, which lead to uniformly most accurate unbiased (UMAU) two-sided confidence intervals. We  show that our results can be applied to distribution-free hypothesis tests for continuous data. Our simulation results demonstrate that all our tests have exact type I error, and are more powerful than current techniques.
\end{abstract}

\noindent{\bf Keywords}: Bernoulli, Hypothesis Test, Confidence Interval, Frequentist, Statistical Disclosure Control, Neyman-Pearson, Confidence Distribution
\pagebreak
\tableofcontents
\pagebreak
\section{Introduction}\label{Introduction}

Differential Privacy (DP), introduced by \citet{Dwork2006:Sensitivity}, offers a rigorous measure of disclosure risk and more broadly, a formal privacy framework.  To satisfy DP, a procedure cannot be a deterministic function of the sensitive data, but must incorporate additional randomness, beyond sampling. Subject to the DP constraint, it is natural to search for a procedure which maximizes the utility of the output. Many works address the goal of minimizing the distance between the outputs of the randomized DP procedure and standard non-private algorithms, but few attempt to  infer properties about the underlying population (for notable exceptions, see related work), which is typically the goal in statistics and scientific research. 
In this paper, we focus on the setting where each individual contributes a sensitive binary value, and we wish to infer the population proportion via hypothesis tests and confidence intervals, subject to DP.
In particular, our main results are focused on deriving {\em uniformly most powerful} (UMP) and {\em uniformly most powerful unbiased} (UMPU) tests, related p-values, and confidence intervals, which optimize finite sample performance. While these tests are designed for binary data, we also show that they can be used to construct certain hypothesis tests for continuous data as well.

UMP tests are fundamental to classical statistics, being closely linked to  sufficiency, likelihood inference, and confidence sets. However, finding UMP tests can be hard and in many cases they do not even exist (see \citealp[Section 4.4]{Schervish1996}). Our results are the first to achieve  UMP tests under $(\epsilon,\delta)-$DP, and are among the first steps towards a general theory of optimal inference under DP. 

{\bfseries \large  Related work} \citet{Vu2009} were the first to perform classical hypothesis tests under DP. They develop private tests for population proportions as well as for independence in $2\times 2$ contingency tables. In both settings, they fix the noise adding distribution, and use approximate sampling distributions to perform these DP tests. A similar approach is used by \citet{Solea2014} to develop tests for normally distributed data. The work of \citet{Vu2009} is extended by \citet{Wang2015:Revisiting} and \citet{Gaboardi2016}, developing additional tests for multinomial data. To implement their tests, \citet{Wang2015:Revisiting} develop asymptotic sampling distributions, verifying via simulations that the type I errors are reliable.  On the other hand, \citet{Gaboardi2016} use simulations to compute an empirical type I error. \citet{Uhler2013} develop DP chi-squared tests and $p$-values for GWAS data, and derive the exact sampling distribution of their noisy statistic. Working under ``Local Differential Privacy,'' a stronger notion of privacy than DP, \citet{Gaboardi2018} develop multinomial tests based on asymptotic distributions. Given a DP output, \citet{Sheffet2017} and \citet{Barrientos2017} develop significance tests for regression coefficients. Following a common strategy in the field of Statistics, \citet{Wang2018} develop approximating distributions for DP statistics, which can be used to construct hypothesis tests and confidence intervals. In a recent work, \citet{Canonne2018} show that for simple hypothesis tests, a DP test based on a clamped likelihood ratio test achieves optimal sample complexity.


Outside the hypothesis testing setting, there is some additional work on optimal population inference under DP. \citet{Duchi2017} give general techniques to derive minimax rates under local DP, and in particular give minimax optimal point estimates for the mean, median, generalized linear models, and nonparametric density estimation.  \citet{Karwa2017} develop nearly optimal confidence intervals for normally distributed data with finite sample guarantees, which could potentially be inverted to give approximately UMP unbiased tests.

Related work on developing optimal DP mechanisms for general loss functions such as \citet{Geng2016} and \citet{Ghosh2009}, give mechanisms that optimize  symmetric convex loss functions, centered at a real-valued statistic.
Similarly, \citet{Awan2019:Structure} derive optimal mechanisms  among the class of $K$-Norm Mechanisms.


{\bfseries \large Our contributions}
The previous literature on DP hypothesis testing has a few characteristics in common: 1) nearly all of these  proposed methods first add noise to the data, and perform their test as a post-processing procedure, 2) all of the hypothesis tests use either asymptotic distributions or simulations to derive approximate decision rules, and 3) while each procedure is derived intuitively based on classical theory, none show that they are optimal among all possible DP algorithms.

In contrast, in this paper we search over all DP hypothesis tests at level $\al$, deriving the \emph{uniformly most powerful} (UMP) test for a population proportion. 
We find that our DP-UMP test can be stated as a post-processing of a noisy statistic, which allows us to efficiently compute exact $p$-values, confidence intervals, and confidence distributions as post-processing.

Sections \ref{Background}-\ref{pValues} appeared in an earlier version of this work \citet{Awan2018:Binomial}, and focus on developing DP-UMP simple and one-sided tests for binomial data. In Section \ref{Setup},  we show that arbitrary DP hypothesis tests, which report `Reject' or `Fail to Reject', can be written in terms of linear inequalities. In Theorem \ref{SuffStatThm}, we show that for exchangeable data, DP tests need only depend on the empirical distribution. We use this structure to find closed-form DP-UMP tests for simple hypotheses in Theorems
\ref{UMP1} and \ref{UMP2}, and extend these results to obtain one-sided DP-UMP tests in Corollary \ref{OneSide1}. These tests are closely tied to our proposed ``\emph{Truncated-Uniform-Laplace}'' (\emph{Tulap}) distribution, which  extends both the discrete Laplace distribution (studied in \citet{Ghosh2009}), and the Staircase distribution of \citet{Geng2016} to the setting of $(\ep,\de)$-DP. We prove that the Tulap distribution satisfies $(\ep,\de)$-DP in Theorem \ref{TulapDP}. While the tests developed in the previous sections only result in the output `Reject' or `Fail to Reject', in Section \ref{pValues}, we show that our DP-UMP tests can be stated as a post-processing of a Tulap random variable. From this formulation, we obtain exact $p$-values via Theorem \ref{PValueResult} and Algorithm \ref{PValueAlgorithm} which agree with our DP-UMP tests.  
In fact, since releasing the Tulap summary statistic satisfies $(\ep,\de)$-DP, we can also produce two-sided $p$-values, confidence intervals, and confidence distributions all in terms of the private summary statistic, thus offering a comprehensive DP statistical analysis for binomial data.

To go beyond the simple tests and one-sided hypothesis results of \citet{Awan2018:Binomial}, we use a Bonferroni correction for multiple comparisons and the one-sided DP-UMP tests to construct two-sided tests, which we detail in Section \ref{s:Bonferroni}. In Section \ref{s:UMPU}, we study unbiased tests for two-sided hypotheses and derive a two-sided DP-UMPU test, using a similar techniques as in Sections \ref{SectionDz} and \ref{SectionDnz}. While these unbiased tests often do not have a convenient form, we show that a close approximation can be used to efficiently compute $p$-values in Section \ref{s:ApproxUnbiased}. In Section \ref{s:CI} we propose methods to construct DP confidence intervals for binomial data. We show in Section \ref{s:OneSide} that our one-sided DP-UMP tests give uniformly most accurate one-sided confidence intervals. In Section \ref{s:TwoSide}, we show that the DP-UMPU test, Bonferroni test and the approximately unbiased two-sided tests can all be used to construct two-sided DP confidence intervals. Furthermore, we show that the DP-UMPU test leads to uniformly most accurate unbiased confidence intervals. In Section \ref{s:CD}, we derive stochastically optimal confidence distributions in terms of the one-sided DP-UMP tests. In Section \ref{DistributionFree}, we apply our results to develop private distribution-free hypothesis tests of continuous data.

In Section \ref{s:Simulations}, we study each of our proposed tests and confidence intervals through simulations. In Section \ref{s:Simulations1}, we verify that our UMP tests have exact type I error, and are more powerful than current techniques. In Section \ref{s:Simulations2}, we compare our different proposed two-sided tests, and in Section \ref{s:Simulations3} we study our two-sided confidence intervals. We conclude in Section \ref{s:Discussion} with discussion.  Detailed proofs and technical lemmas are postponed to Section \ref{s:Appendix}, which is in the supplementary material.

 \section{Hypothesis testing}\label{s:HT}
\subsection{Background and notation}\label{Background}
We use capital letters to denote random variables and lowercase letters for particular values. For a random variable $X$, we denote $F_X$ as its cumulative distribution function (cdf), $f_X$ as either its probability density function (pdf) or probability mass function (pmf), depending on the context.

For any set $\mscr X$, the $n$-fold cartesian product of $\mscr X$ is $\mscr X^n = \{(x_1,x_2,\ldots, x_n) \mid x_i \in \mscr X\}$.  We denote elements of $\mscr X^n$ with an underscore to emphasize that they are vectors. The \emph{Hamming distance} metric on $\mscr X^n$ is $H: \mscr X^n \times \mscr X^n \rightarrow \ZZ^{\geq 0}$, defined by $H(\ul x,\ul x') =\# \{i \mid x_i\neq x'_i\}$. 

Differential Privacy, introduced by \citet{Dwork2006:Sensitivity}, provides a formal measure of disclosure risk. The notion of DP that we give in Definition \ref{DPDefn} more closely resembles the formulation in \citet{Wasserman2010:StatisticalFDP}, which uses the language of distributions rather than random mechanisms. It is important to emphasize that the notion of Differential Privacy in Definition \ref{DPDefn} does not involve any distribution model on $\mscr X^n$.

\begin{defn}
  [Differential Privacy: \citet{Dwork2006:Sensitivity,Wasserman2010:StatisticalFDP}]\label{DPDefn}
Let $\ep>0$, $\de\geq 0$, and $n\in \{1,2,\ldots\}$ be given. Let $\mscr X$ be any set, and $(\mscr Y, \mscr F)$ be a measurable space. Let $\mscr P = \{P_{\ul x} \mid \ul x\in \mscr X^n\}$ be a set of probability measures on $(\mscr Y, \mscr F)$. We say that $\mscr P$ satisfies \emph{$(\ep, \de)$-Differential Privacy} ($(\ep,\de)$ - DP) if for all $B\in \mscr F$ and all $\ul x,\ul x' \in \mscr X^n$ such that $H(\ul x,\ul x')=1$, we have $P_{\ul x}(B) \leq e^\ep P_{\ul x'}(B) + \de.$
\end{defn}
In Definition \ref{DPDefn}, we interpret $\ul x\in \mscr X^n$ as the database we collect, where $\mscr X$ is the set of possible values that one individual can contribute, and $Y \sim P_{\ul x}$ as the statistical result we report to the public.
With this interpretation, if a set of distributions satisfies $(\ep,\de)$-DP for small values of $\ep$ and $\de$, then if one person's data is changed in the database, the change in the distribution of $Y$ is small. Ideally $\ep$ is a value less than $1$, and $\de\ll \frac 1n$ allows us to disregard events which have small probability. A special case is when $\de=0$, and $(\ep,0)$-DP is referred to as pure DP.

One of our main goals in this paper is to find uniformly most powerful (UMP) hypothesis tests, subject to DP. As the output of a DP method is necessarily a random variable, we work with randomized hypothesis tests, which we review in Definition \ref{HT}. Our notation follows that of \citet[Chapter 4]{Schervish1996}.

\begin{defn}
  [Hypothesis Test]\label{HT}
  Let $(X_1,\ldots, X_n)\in \mscr X^n$ be distributed $X_i \iid f_\ta$
  , where $\ta\in \Ta$. Let $\Ta_0, \Ta_1$ be disjoint subsets of $\Ta$. We call  $\Ta_0$ the \emph{null } and $\Ta_1$ the \emph{alternative}. A \emph{(randomized) test} of $H_0: \ta \in \Ta_0$ versus $H_1: \ta\in \Ta_1$ is a measurable function $\phi: \mscr X^n \rightarrow [0,1]$. We say a test $\phi$ is at \emph{level} $\al$ if $\sup_{\ta\in \Ta_0} \EE_{f_\ta} \phi \leq \al$, and at \emph{size} $\al$ if $\sup_{\ta\in \Ta_0} \EE_{f_\ta} \phi = \al$. 
  The \emph{power} of $\phi$ at $\ta$ is denoted $\beta_\phi(\ta) = \EE_{f_\ta}\phi$. 

Let $\Phi$ be a set of tests. We say that $\phi^*\in \Phi$ is the \emph{uniformly most powerful level $\al$} (UMP-$\al$) test among $\Phi$ for $H_0: \ta \in \Ta_0$ versus $H_1:\ta\in \Ta_1$ if 1) $\sup_{\ta\in \Ta_0}\beta_{\phi^*}(\ta)\leq \al$ and 2) for any $\phi \in \Phi$ such that $\sup_{\ta\in \Ta_0} \beta_{\phi}(\ta)\leq \al$ we have 
$\beta_{\phi^*}(\ta) \geq \beta_{\phi}(\ta)$, for all $\ta\in \Ta_1$.
\end{defn}

In Definition \ref{HT}, $\phi(x)$ is the probability of rejecting the null hypothesis, given that we observe $x\in \mscr X^n$. That is, the output of a test is either `Reject', or `Fail to Reject' with respective probabilities $\phi(x)$, and $1-\phi(x)$. While the condition of $(\ep,\de)$-DP does not involve the randomness of $X$, for hypothesis testing, the level/size, and power of a test depend on the model for $X$. In Section \ref{Setup}, we study the set of hypothesis tests which satisfy $(\ep,\de)$-DP.

\subsection{Problem setup and exchangeability condition}\label{Setup}

We begin this section by considering arbitrary hypothesis testing problems under DP.
Let $\phi: \mscr X^n \rightarrow [0,1]$ be any test. Since the only possible outputs of the mechanism are `Reject' or `Fail to Reject' with probabilities $\phi(\ul x)$ and $1-\phi(\ul x)$, the test $\phi$ satisfies $(\ep, \de)$-DP if and only if for all $\ul x,\ul x'\in \mscr X^n$ such that $H(\ul x,\ul x')=1$,
\begin{equation}
  \label{DPBinary}
\phi(\ul x) \leq e^\ep \phi(\ul x') + \de\quad \text{and} \quad (1- \phi(\ul x)) \leq e^\ep (1-\phi(\ul x')) + \de.
\end{equation}

\begin{remark}
 For any simple hypothesis test, where $\Phi_0$ and $\Phi_1$ are both singleton sets, the DP-UMP test $\phi^*$ is the solution to a linear program. If $\mscr X$ is finite, this observation allows one to explore the structure of DP-UMP tests through numerical linear program solvers. 
\end{remark}

  Given the random vector $\ul X\in \mscr X^n$, initially it may seem that we need to consider all $\phi$, which are arbitrary functions of $\ul X$. However, assuming that $\ul X$ is exchangeable, Theorem \ref{SuffStatThm} below says that for any DP hypothesis tests, we need only consider tests which are functions of the empirical distribution of $\ul X$. In other words, $\phi$ need not consider the order of the entries in $\ul X$. This result is reminiscent of De Finetti's Theorem (see \citealp[Theorem 1.48]{Schervish1996}) in classical statistics.
  \begin{thm}
    \label{SuffStatThm}
    Let $\Ta$ be a set and $\{\mu_\ta\}_{\ta\in \Ta}$ be a set of  exchangeable  distributions on $\mscr X^n$. Let $\phi: \mscr X^n \rightarrow [0,1]$ be a test satisfying \eqref{DPBinary}. Then there exists $\phi': \mscr X^n \rightarrow [0,1]$ satisfying \eqref{DPBinary} which only depends on the empirical distribution of $X$, such that $\int \phi'(\ul x) \ d\mu_\ta = \int \phi(\ul x) \ d\mu_\ta$, for all $\ta\in \Ta$.

\end{thm}
\begin{proof}[Proof.]
  Define $\phi'$ by $\phi'(\ul x) = \frac{1}{n!} \sum_{\pi \in \sa(n)} \phi(\pi(\ul x))$, where $\sa(n)$ is the symmetric group on $n$ letters. For any $\pi \in \sa(n)$, $\phi(\pi(\ul x))$ satisfies $(\ep,\de)$-DP. By exchangeability, $\int \phi(\pi(\ul x))\ d\mu_\ta = \int \phi(\ul x)\ d\mu_\ta$. Since condition \ref{DPBinary} is closed under convex combinations, and integrals are linear, the result follows.
\end{proof}

We now state the particular problem which is the primary focus of the remainder of Section \ref{s:HT}. Each individual contributes a sensitive binary value to the database, and the database can be thought of as a random vector $\ul X \in \{0,1\}^n$, where  $X_i$ represents the sensitive data of individual $i$. We model $\ul X$ as $X_i \iid \mathrm{Bern}(\ta)$, where $\ta$ is unknown. Then the statistic $X= \sum_{i=1}^n  X_i\sim \mathrm{Binom}(n,\ta)$ encodes the empirical distribution of $\ul X$.
%
%
By Theorem \ref{SuffStatThm}, we can restrict our attention to tests which are functions of $X$. Such tests $\phi:\{0,1,\ldots, n\}\rightarrow[0,1]$ satisfy $(\ep,\de)$ -DP  if and only if for all $x\in \{1,2,\ldots, n\}$,

\begin{align}
  \phi(x) &\leq e^\ep \phi(x-1)+\de\label{DP1}\\
  \phi(x-1)&\leq e^\ep \phi(x)+\de\label{DP2}\\
(1-\phi(x))&\leq e^\ep(1-\phi(x-1))+\de\label{DP3}\\
(1-\phi(x-1))&\leq e^\ep(1-\phi(x))+\de\label{DP4}.
\end{align}
We denote the set of all tests which satisfy \eqref{DP1}-\eqref{DP4} as  
$\mscr D^n_{\ep,\de} =\big\{\phi:\phi \text{ satisfies \eqref{DP1}-\eqref{DP4}}\big\}.$

\begin{remark}
  For arbitrary DP hypothesis testing problems, the number of constraints generated by \eqref{DPBinary} could be very large, even infinite, but for our problem we only have $4n$ constraints.
\end{remark}

\subsection{Simple DP-UMP  tests when $\de=0$}\label{SectionDz}

In this section, we derive the DP-UMP test when $\de=0$ for simple hypotheses. In particular, given $n, \ep>0, \al>0, \ta_0<\ta_1$, and $X\sim \mathrm{Binom}(n,\ta)$, we find the UMP test at level $\al$ among $\mscr D^n_{\ep,0}$ for testing  $H_0:\ta = \ta_0$ versus $H_1: \ta=\ta_1$.

Before developing these tests, we introduce the \emph{Truncated-Uniform-Laplace} (Tulap) distribution, defined in Definition \ref{TulapDefn}, which is central to all of our main results. To motivate  this distribution, recall that \citet{Geng2016} show for general loss functions, adding discrete Laplace noise $L\sim \mathrm{DLap}(e^{-\ep})$ to $X$ is optimal  under $(\ep,0)$-DP. For this reason, it is natural to consider a test which post-processes $X+L$. However, we know by classical UMP theory that since $X+L$ is discrete, a randomized test is required. Instead of using a randomized test, by adding uniform noise $U\sim \mathrm{Unif}(-1/2,1/2)$ to $X+L$, we obtain a continuous sampling distribution, from which a deterministic test is available. We call the distribution of $(X+L+U)\mid {X}$ as $\mathrm{Tulap}(X,b,0)$. The distribution $\mathrm{Tulap}(X,b,q)$ is obtained by truncating
 between the $(q/2)^{th}$ and $(1-q/2)^{th}$ quantiles  
 of $\mathrm{Tulap}(X,b,0)$.

In Definition \ref{TulapDefn}, we use the \emph{nearest integer function} $[\cdot]:\RR\rightarrow \ZZ$. For any real number $t\in \RR$, $[t]$ is defined to be the integer nearest to $t$. If there are two distinct integers which are nearest to $t$, we take $[t]$ to be the even one. Note that, $[-t]=-[t]$ for all $t\in \RR$.


\begin{defn}[Truncated-Uniform-Laplace (Tulap)]\label{TulapDefn}
  Let $N$ and $N_0$ be real-valued random variables. Let $m\in \RR$, $b\in(0,1)$ and $q\in [0,1)$. We say that $N_0 \sim \mathrm{Tulap}(m,b,0)$ and $N\sim \mathrm{Tulap}(m,b,q)$ if $N_0$ and $N$ have the following cdfs:
\[F_{N_0}(x) = \begin{cases}
\frac{b^{-[x-m]}}{1+b} \l(b+(x-m-[x-m]+\frac 12)(1-b)\r)&\text{ if }x\leq [m]\\
1-\frac{b^{[x-m]}}{1+b}\l(b+ ([x-m]-(x-m)+\frac 12 ) (1-b)\r) &\text{ if } x>[m],
\end{cases}\]
 \[F_N(x) = \begin{cases}
     0&\text{ if } F_{N_0}<q/2\\
     \frac{F_{N_0}(x) - \frac q2}{1-q}&\text{ if } \frac q2\leq F_{N_0}(x) \leq 1-\frac q2\\
     1&\text{ if } F_{N_0}>1-\frac q2.
   \end{cases}\]
\end{defn}

Note that a Tulap random variable $\mathrm{Tulap}(m,b,q)$ is continuous and symmetric about $m$.

\begin{remark} 
  The Tulap distribution extends the staircase and discrete Laplace distributions as follows: $\mathrm{Tulap}(0,b,0) \overset d = \mathrm{Staircase}(b,1/2)$ and $[\mathrm{Tulap}(0,b,0)] \overset d= \mathrm{DLap}(b)$, where $\mathrm{Staircase}(b,\ga)$ is the distribution in \citet{Geng2016}. 
\citet{Geng2016} show that for a real valued statistic $T$ and convex symmetric loss functions centered at $T$, the optimal noise distribution for $(\ep,0)$-DP is $\mathrm{Staircase}(b,\ga)$ for $b=e^{-\ep}$ and some $\ga\in (0,1)$. If the statistic is a count, then \citet{Ghosh2009} show that $\mathrm{DLap}(b)$ is optimal. Our results agree with these works when $\de=0$, and extend them to the case of arbitrary $\de$.
\end{remark}

Now that we have defined the Tulap distribution, we are ready to develop the UMP test among $\mscr D^n_{\ep,0}$ for the simple hypotheses $H_0:\ta = \ta_0$ versus $H_1: \ta=\ta_1$. In classical statistics, the UMP for this test is given by the \emph{Neyman-Pearson lemma}, however in the DP framework, our test must satisfy \eqref{DP1}-\eqref{DP4}. Within these constraints, we follow the logic behind the Neyman-Pearson lemma as follows.
Let $\phi\in \mscr D_{\ep,0}^n$. Thinking of $\phi(x)$ defined recursively, equations \eqref{DP1}-\eqref{DP4} give  upper and  lower bounds for $\phi(x)$ in terms of $\phi(x-1)$. Since $\ta_1>\ta_0$, and binomial distributions have a monotone likelihood ratio (MLR) in $X$,  larger values of $X$ give more evidence for $\ta_1$ over $\ta_0$. Thus, $\phi(x)$ should be increasing in $x$ as much as possible, subject to  \eqref{DP1}-\eqref{DP4}. Lemma \ref{RecurrenceLem1} shows that taking $\phi(x)$ to be such a function is equivalent to having $\phi(x)$ be the cdf of a Tulap random variable.

\begin{lem}\label{RecurrenceLem1}
  Let $\ep>0$ be given. Let $\phi:\{0,1,2,\ldots, n\} \rightarrow (0,1)$. The following are equivalent:
  \begin{enumerate}[1)]
  \item There exists $m\in (0,1)$ such that for $x=0,\ldots, n$,
    \[\phi(x)=\begin{cases}
        m&\text{if }x=0\\
        \min\{e^\ep \phi(x-1) , 1-e^{-\ep} (1-\phi(x-1))\}&\text{if }x>0.
        \end{cases}\]
    \item There exists $m\in (0,1)$ such that 
      for $x=0,\ldots, n$,\\
  \[\phi(x) = \begin{cases}
      m&\text{if }x=0\\
e^\ep \phi(x-1) & \text{if } x>0\text{ and }\phi(x-1) \leq \frac{1}{1+e^\ep} \\
1-e^{-\ep}(1-\phi(x-1))&\text{if }x>0\text{ and }\phi(x-1)>\frac{1}{1+e^\ep}.
\end{cases}\]
\item There exists $m\in \RR$ such that $\phi(x) = F_{N_0}(x-m)$ for $x=0,1,2,\ldots, n$, where $N_0\sim \mathrm{Tulap}(0,b=e^{-\ep},0)$.
  \end{enumerate}
\end{lem}
\begin{proof}[Proof Sketch.]
First show that 1) and 2) are equivalent by checking which constraint is active. Then verify that $F_{N_0}(x-m)$ satisfies the recurrence of 2). This can be done using the properties of the Tulap cdf, stated in Lemma \ref{CLem}, found in Section \ref{s:Appendix}.
\end{proof}

While the form of 1) in Lemma \ref{RecurrenceLem1} is intuitive, the connection to the Tulap cdf in 3) allows for a usable closed-form of the test. This connection with the Tulap distribution is crucial for the development in Section \ref{pValues}, which shows that the test in Lemma \ref{RecurrenceLem1} can be achieved by post-processing $X+N$, where $N$ is distributed as Tulap.

  It remains to show that the tests in Lemma \ref{RecurrenceLem1} are in fact UMP among $\mscr D_{\ep,0}^n$. The main tool used to prove this is Lemma \ref{BingLem}, which is a standard result in the classical hypothesis testing theory.

\begin{lem}
\label{BingLem}
  Let $(\mscr X, \mscr F,\mu)$ be a measure space and let $f$ and $g$ be two densities on $\mscr X$ with respect to $\mu$. Suppose that $\phi_1,\phi_2: \mscr X\rightarrow [0,1]$ are such that $\int \phi_1 f \ d\mu \geq \int \phi_2 f \ d\mu$, and there exists $k\geq 0$ such that $\phi_1\geq \phi_2$ when $g\geq kf$ and $\phi_1\leq \phi_2$ when $g< k f$. Then $\int \phi_1 g \ d\mu \geq \int \phi_2 g \ d\mu$.
\end{lem}
\begin{proof}
  Note that $(\phi_1 - \phi_2)(g-kf)\geq 0$ for almost all $x\in \mscr X$ (with respect to $\mu$). This implies that $\int (\phi_1 - \phi_2)(g-kf) \ d\mu \geq 0$. Hence, $\int \phi_1 g \ d\mu - \int \phi_2 g \ d\mu \geq k\l(\int \phi_1 f\ d\mu - \int \phi_2 f \ d\mu\r) \geq 0$.
\end{proof}

Next we present our key result, Theorem \ref{UMP1}, which can be viewed as a `Neyman-Pearson lemma' for binomial data under $(\ep,0)$-DP. We extend this result in Theorem \ref{UMP2} for $(\ep,\de)$-DP. 

\begin{thm}\label{UMP1}
  Let $\ep>0$, $\al\in (0,1)$, $0\leq\ta_0<\ta_1\leq 1$, and $n\geq1$ be given. Observe $X\sim \mathrm{Binom}(n,\ta)$, where $\ta$ is unknown.  Set the decision rule $\phi^*: \ZZ\rightarrow [0,1]$ by $\phi^*(x) = F_{N_0}(x-m)$, where $N_0 \sim \mathrm{Tulap}(0,b=e^{-\ep},0)$ and  $m$ is chosen such that $E_{\ta_0}\phi^*(x)=\al$. Then $\phi^*$ is UMP-$\al$ test of $H_0: \ta=\ta_0$ versus $H_1: \ta = \ta_1$ among $\mscr D_{\ep,0}^n$.
\end{thm}
\begin{proof}[Proof Sketch.]
  Let $\phi$ be any other test which satisfies \eqref{DP1}-\eqref{DP4} at level $\al$. Then, since $\phi^*$ can be written in the form of 1) in Lemma \ref{RecurrenceLem1}, there exists $y\in \ZZ$ such that $\phi^*(x)\geq \phi(x)$ when $x\geq y$ and $\phi^*(x)\leq \phi(x)$ when $x<y$. By MLR of the binomial distribution and Lemma \ref{BingLem}, we have $\beta_{\phi^*}(\ta_1)\geq \beta_{\phi}(\ta_1)$.
\end{proof}

While the classical Neyman-Pearson lemma results in an acceptance and rejection region, the DP-UMP always has some probability of rejecting the null, due to the constraints \eqref{DP1}-\eqref{DP4}. As $\ep\uparrow \infty$, the DP-UMP converges to the non-private UMP.

\subsection{Simple and one-sided DP-UMP  tests when $\de\geq0$}\label{SectionDnz}
In this section, we extend the results of Section \ref{SectionDz} to allow for $\de\geq 0$. We begin by proposing the form of the DP-UMP test for simple hypotheses. As in Section \ref{SectionDz}, the DP-UMP test is increasing in $x$ as much as \eqref{DP1}-\eqref{DP4} allow. 
Lemma \ref{RecurrenceLem2} states that such a test can be written as the cdf of a Tulap random variable, where the parameter $q$ depends on $\ep$ and $\de$. We omit the proof of Theorem \ref{UMP2}, which mimics the proof of Theorem \ref{UMP1}.

\begin{lem}\label{RecurrenceLem2}
  Let $\ep>0$ and $\de\geq 0$ be given and set $b=e^{-\ep}$ and $q = \frac{2\de b}{1-b+2\de b}$. Let $\phi:\{0,1,2,\ldots, n\}\rightarrow [0,1]$. The following are equivalent:
  \begin{enumerate}[1)]
  \item There exists $y\in \{0,1,2,\ldots, n\}$ and $m \in (0,1)$ such that for $x=0,\ldots, n$, 
     \[
    \phi(x) = \begin{cases}
0&\text{if }x<y\\
m&\text{if }x=y\\
\min\{e^{\ep} \phi(x-1)+\de,\quad 1-e^{-\ep}(1-\phi(x-1))+e^{-\ep}\de,\quad 1\} &\text{if }x>y.
 \end{cases}\]
\item There exists $y\in \{0,1,2,\ldots, n\}$ and $m\in (0,1)$ such that for $x=0,\ldots, n$,
\[\phi(x) = \begin{cases}
0&\text{if }x<y\\
m&\text{if }x=y\\
e^\ep \phi(x-1)+\de&\text{ if }x>y \text{ and }\phi(x-1) \leq \frac{1-\de}{1+e^{\ep}}\\
1-e^{-\ep}(1-\phi(x-1))+e^{-\ep}\de&\text{ if }x>y \text{ and }\frac{1-\de}{1+e^{\ep}} \leq \phi(x-1) \leq 1-\de\\
1&\text{ if }x>y \text{ and }\phi(x-1) >1-\de.
\end{cases}\]
\item There exists $m\in \RR$ such that 
$\phi(x) = F_N(x-m)$ where $N\sim \mathrm{Tulap}(0,b,q)$.
  \end{enumerate}
\end{lem}
\begin{proof}[Proof Sketch.]
  The equivalence of 1) and 2) only requires determining which constraints are active. To show the equivalence of 2 and 3, we verify that $F_{N}(x-m)$ satisfies the recurrence of 2), using the expression of  $F_N(x)$ in terms of $F_{N_0}(x)$ given in Definition \ref{TulapDefn}, and the results of 
  Lemma \ref{RecurrenceLem1}.
\end{proof}

\begin{thm}\label{UMP2}
  Let $\ep>0$, $\de\geq 0$, $\al \in (0,1)$, $0\leq\ta_0<\ta_1\leq 1$, and $n\geq 1$ be given. Observe $X\sim \mathrm{Binom}(n,\ta)$, where $\ta$ is unknown. Set $b=e^{-\ep}$ and $q = \frac{2\de b}{1-b+2\de b}$. Define $\phi^*: \ZZ \rightarrow [0,1]$ by $\phi^*(x) = F_N(x-m)$ where $N\sim \mathrm{Tulap}(0,b,q)$ and $m$ is chosen such that $E_{\ta_0}\phi^*(x)=\al$. Then $\phi^*$ is UMP-$\al$ test of $H_0: \ta=\ta_0$ versus $H_1: \ta = \ta_1$ among $\mscr D_{\ep,\de}^n$.
\end{thm}

So far we have focused on simple hypothesis tests, but since our test only depends on $\ta_0$, and not on $\ta_1$, our test is in fact the DP-UMP for one-sided tests, as stated in Corollary \ref{OneSide1}. Corollary \ref{OneSide1} also shows that we can use our tests to build DP-UMP tests for $H_0: \ta\geq \ta_0$ versus $H_1: \ta<\ta_0$ as well. Hence, Corollary \ref{OneSide1} is our most general result so far, containing Theorems \ref{UMP1} and \ref{UMP2} as special cases.

\begin{cor}\label{OneSide1}
Let $X\sim \mathrm{Binom}(n,\ta)$. Set $\phi^*(x) = F_N(x-m_1)$ and $\psi^*(x) = 1-F_N(x-m_2)$, where $N\sim \mathrm{Tulap}\l(0, b=e^{-\ep},q=\frac{2\de b}{1-b+2\de b}\r)$ and $m_1,m_2$ are chosen such that $E_{\ta_0}\phi^*(x)=\al$ and $E_{\ta_0} \psi^*(x)=\al$. Then $\phi^*(x)$ is UMP-$\al$ among $\mscr D_{\ep,\de}^n$ for testing $H_0: \ta\leq \ta_0$ versus $H_1: \ta>\ta_0$, and $\psi^*(x)$ is UMP-$\al$ among $\mscr D_{\ep,\de}^n$ for testing $H_0: \ta\geq \ta_0$ versus $H_1: \ta<\ta_0$.
\end{cor}

\subsection{Optimal one-sided private p-values}\label{pValues}
For the  DP-UMP tests developed in Sections \ref{SectionDz} and \ref{SectionDnz}, the output  is simply to `Reject' or `Fail to Reject' $H_0$. In scientific research, however, $p$-values are often used to weigh the evidence in favor of the alternative hypothesis over the null. Intuitively, a $p$-value is the smallest level $\al$, for which a test outputs `Reject'. Definition \ref{def:pvalue} gives a formal definition of a $p$-value.

\begin{defn}[p-Value: \citet{Casella2002}]\label{def:pvalue}
  For a random vector $X_i \iid f_\ta$, a \emph{p-value} for $H_0: \ta\in \Ta_0$ versus $H_1: \ta\in \Ta_1$ is a statistic $p(\ul X)$ taking values in $[0,1]$, such that for every $\al \in [0,1]$,
  \[\sup_{\ta\in \Ta_0} P_\ta(p(X)\leq \al) \leq \al.\]
The smaller the value of $p(X)$, the greater evidence we have for $H_1$ over $H_0$.
\end{defn}

In this section, we show that our  proposed DP-UMP tests can be achieved by post-processing a Tulap random variable. Using this, we develop a differentially private algorithm for releasing a private $p$-value which agrees with the DP-UMP tests in Sections \ref{SectionDz} and \ref{SectionDnz}. While we state our $p$-values for one-sided tests, they also apply to simple tests as a special case.



Since our DP-UMP test from Theorem \ref{UMP2} rejects  with probability $\phi^*(x) = F_N(x-m)$, given $N\sim F_N$, $\phi^*(x)$ rejects the null if and only if $X+N\geq m$. So, our DP-UMP tests can be stated as a 
post-processing of $X+N$. Theorem \ref{TulapDP} states that releasing $X+N$ satisfies $(\ep,\de)$-DP. 
By the post-processing property of DP (see \citealp[Proposition 2.1]{Dwork2014:AFD}),
once we release $X+N$, any function of $X+N$ also satisfies $(\ep,\de)$-DP. Thus, we can compute our private UMP-$\al$ tests as a function of $X+N$ for any $\al$. 
 The smallest $\al$ for which we reject the null is the $p$-value for that test. In fact Algorithm \ref{PValueAlgorithm} and Theorem \ref{PValueResult} give a more elegant method of computing this $p$-value.

\begin{thm}\label{TulapDP}
Let $\mscr X$ be any set, and $T:\mscr X^n\rightarrow \ZZ$, with $\Delta(T) = \sup |T(\ul x)-T(\ul x')|=1$, where the supremum is over the set $\{(\ul x,\ul x')\in \mscr X^n\times \mscr X^n\mid H(\ul x,\ul x')=1\}$. Then the set of distributions $\l\{\mathrm{Tulap}\l(T(\ul x),b=e^{-\ep}, \frac{2\de b}{1-b+2\de b}\r)\middle| \ul x\in\mscr X^n\r\}$ satisfies $(\ep,\de)$-DP.
\end{thm}
\begin{proof}[Proof Sketch.]
Since Tulap random variables are continuous and have MLR in $T(\ul x)$, by Lemma \ref{CDFLem} in Section \ref{s:Appendix}, it suffices to show that for all $t\in \RR$, the cdf of a Tulap random variable $F_{N}(t-T(\ul x))$ satisfies \eqref{DPBinary}, with $\phi(\ul x)$ replaced with $F_{N}(t-T(\ul x))$. This already established in Lemma \ref{RecurrenceLem2}, by the equivalence of 1) and 3).
\end{proof}

\begin{thm}\label{PValueResult}
  Let $\ep>0$, $\de\geq 0$, $X\sim \mathrm{Binom}(n,\ta)$ where $\ta$ is unknown, and $Z|X \sim \mathrm{Tulap}(X,b=e^{-\ep}, q= \frac{2\de b}{1-b+2\de b})$. Then 
  \begin{enumerate}[1)]
  \item $p(\ta_0,Z) \defeq P(X+N\geq Z\mid Z)$ is a $p$-value for $H_0: \ta\leq \ta_0$ versus $H_1: \ta>\ta_0$, where the probability is over $X\sim \mathrm{Binom}(n,\ta_0)$ and $ N\sim \mathrm{Tulap}(0,b,q)$.
  \item Let $0<\al<1$ be given. The test $\phi^*(x) = P_{Z \sim \mathrm{Tulap}(x,b,q)} (p(\ta_0,Z)\leq \al \mid X)$ is UMP-$\al$ for $H_0: \ta\leq \ta_0$ versus $H_1: \ta>\ta_0$ among $\mscr D_{\ep,\de}^n$.
  \item  For all $\ta_1>\ta_0$, $p(\ta_0,Z)$ is the stochastically smallest $(\ep,\de)$-DP  $p$-value for $H_0: \ta\leq \ta_0$ versus $H_1: \ta\geq \ta_0$.
\item The output of Algorithm \ref{PValueAlgorithm} is equal to $p(\ta_0,Z)$.
  \end{enumerate}
\end{thm}

In the following corollary, we see that $1-p(\ta_0,Z) = P(X+N\leq Z\mid Z)$ is the corresponding $p$-value for $H_0: \ta\geq \ta_0$ versus $H_1: \ta<\ta_0$, with all the analogous properties.
\begin{cor}
In the same setup as Theorem \ref{PValueResult}, $1-p(\ta_0,Z) = P(X+N\leq Z\mid Z)$ is the stochastically smallest $(\ep,\de)$-DP $p$-value for $H_0: \ta\geq \ta_0$ versus $H_1: \ta<\ta_0$, and $\psi^*(x) = P_{Z\sim \mathrm{Tulap}(x,b,q)}(1-p(\ta_0,Z)\leq\alpha\mid X)$  agrees with the UMP-$\al$ test in Corollary \ref{OneSide1}. 
\end{cor}

\begin{algorithm}
\caption{UMP one-sided $p$-value for binomial data under $(\ep,\de)$-DP}
\scriptsize
INPUT: $n\in\NN$, $\ta_0\in (0,1)$, $\ep>0$, $\de\geq 0$,  $Z\sim \mathrm{Tulap}\l(X,b=e^{-\ep},q=\frac{2\de b}{1-b+2\de b}\r)$,
\begin{algorithmic}[1]
  \setlength\itemsep{0em}
  \STATE Set $F_N$ as the cdf of $N\sim \mathrm{Tulap}(0,b,q)$
\STATE Set $\ul F = (F_N(0-Z),F_N(1-Z),\ldots, F_N(n-Z))^\top$
\STATE Set $\ul B = (\binom n0 \ta_0^0(1-\ta_0)^{n-0}, \binom n1 \ta_0^1(1-\ta_0)^{n-1},\ldots, \binom nn \ta_0^n (1-\ta_0)^{n-n})^\top$
\end{algorithmic}
OUTPUT: $\ul F^\top \ul B$
\label{PValueAlgorithm}
\end{algorithm}

To implement Algorithm \ref{PValueAlgorithm}, we must be able to sample a Tulap random variable, which Algorithm \ref{SampleTulap} provides. The algorithm is based on the expression of $\mathrm{Tulap}(m,b,0)$ in terms of geometric and uniform variables, and uses rejection sampling when $q>0$ (see  \citealp[Chapter 11]{Bishop2006} for an introduction to rejection sampling). A detailed proof that the output of this  algorithm follows the correct distribution can be found in Lemma \ref{TulapLem} in Section \ref{s:Appendix}.
\begin{algorithm}
  \caption{Sample from Tulap distribution: $N\sim \mathrm{Tulap}(m,b,q)$}
\scriptsize
INPUT: $m\in \RR$, $b\in (0,1)$, $q\in [0,1)$.
\begin{algorithmic}[1]
  \setlength\itemsep{0em}
  \STATE Draw $G_1,G_2 \iid \mathrm{Geom}(1-b)$ and $U\sim \mathrm{Unif}(-1/2,1/2)$
\STATE Set $N=G_1-G_2+U+m$
\STATE If $F_{N_0}(N)<q/2$ or $F_{N_0}(N)>1-q/2$, where $N_0\sim \mathrm{Tulap}(m,b,0)$,
go to 1:
\end{algorithmic}
OUTPUT: $N$
\label{SampleTulap}
\end{algorithm}

\begin{remark}
  Since we know that releasing $Z=X+N$, where $N$ is a Tulap random variable, satisfies $(\ep,\de)$-DP, one could release $Z$ and compute all of the desired inference quantities as a post-processing of $Z$, at no additional cost to privacy. 
  In the remainder of the paper, we show that private two-sided $p$-values, confidence intervals, and confidence distributions can all be expressed as a post-processing of the summary statistic $Z$, leading to a more complete DP statistical analysis of binomial data. 
\end{remark}

\begin{remark}[Asymptotic Relative Efficiency]
 One may wonder about the asymptotic properties of the DP-UMP test compared to the non-private UMP test. It is not hard to show that 
 for any fixed $\ep>0$, $\de$, and $\ta_0\in (0,1)$, our proposed DP-UMP test has asymptotic relative efficiency (ARE) of 1, relative to the non-private UMP test (see \citealp[Section 14.3]{VanDerVaart2000} for an introduction to ARE).
    Let $X\sim \mathrm{Binom}(n,\ta_0)$. Define the two test statistics as $T_{1}=X$ and $T_{2}=X+N$, where $N\sim \mathrm{Tulap}(0,b,q)$. The ARE of the DP-UMP relative to the non-private UMP test is $(C_{2}/C_{1})^2$, where 
    \begin{align*}
 C_i = \lim_{n\rightarrow \infty} \l(\frac{d}{d\ta} \EE_{\ta} T_i\Big|_{\ta=\ta_0}\r)\Big/\sqrt{n \var_{\ta_0}(T_i)}, \text{ for } i=1,2. 
  \end{align*}
We compute $\EE_\ta T_i = n\ta$, $\var_{\ta_0}(T_1) = n\ta_0(1-\ta_0)$, and $\var_{\ta_0}(T_2) = n\ta_0(1-\ta_0) + \var(N)$. Since $\var(N)$ is a constant,  we have that $C_1=C_2 =(\ta_0(1-\ta_0))^{-1/2}$.
\end{remark}


\subsection{Bonferroni two-sided tests}\label{s:Bonferroni}

In this section as well as in Sections \ref{s:UMPU} and \ref{s:ApproxUnbiased}, we develop ``two-sided'' tests for hypotheses of the form $H_0: \ta =\ta_0$ versus $H_1: \ta\neq \ta_0$. One way of viewing this problem is as a \emph{multiple testing problem}, where we test both $H_0: \ta\leq \ta_0$ and $H_0: \ta\geq \ta_0$. It is well known that if $p_1$ is a $p$-value for $H_0: \ta\leq \ta_0$ and $p_2$ is a $p$-value for $H_0: \ta\geq \ta_0$, then $p = 2\min(p_1,p_2)$ is a $p$-value for $H_0: \ta=\ta_0$.

More generally, if we are interested in testing $H_0: \ta\in \cap_{i=1}^k \Ta_k$, and $p_i$ is a $p$-value for testing $H_0: \ta\in \Ta_k$, then $p = k \min\{p_1,\ldots, p_k\}$ is a $p$-value for  $H_0: \ta\in \cap_{i=1}^k \Ta_k$. We call this setting a \emph{multiple testing problem} or an \emph{intersection-union test} (See \citealp[ Section 8.2.3]{Casella2002}). The factor $k$ in the computation of $p$ is called the \emph{Bonferroni correction}.




\begin{prop}\label{prop:Bonferroni}
  Let $\ep>0$, $\de\geq 0$, $X\sim \mathrm{Binom}(n,\ta)$ where $\ta$ is unknown, and $Z\mid X\sim \mathrm{Tulap}\left(X,b=e^{-\ep},q=\frac{2\de b}{1-b+2\de b}\right)$. Then
  \begin{enumerate}[1)]
  \item $p'(\ta_0,Z) =2 \min(p(\ta_0,Z),1-p(\ta_0,Z))$ is a $(\ep,\de)$-DP $p$-value for $H_0: \ta=\ta_0$ versus $H_0: \ta\neq \ta_0$ for any $\ta_0 \in (0,1)$, where $p$ is the one-sided $p$-value from Theorem \ref{PValueResult}.
  \item The test $\phi'(X) = P_Z(p'(\ta_0,Z)\leq \alpha\mid X)$ is in $\mscr D_{\ep,\de}^n$ and can be written in the form $\phi'(x) = \phi^*(x) + \psi^*(x)$, where $\phi^*,\psi^*$ are as defined in Corollary \ref{OneSide1} at size $\alpha/2$.
  \item The test $\phi'(X)$ is uniformly more powerful than any level $\alpha/2$ test in $\mscr D_{\ep,\de}^n$ for $H_0: \ta=\ta_0$ versus $H_1: \ta\neq \ta_0$.
  \end{enumerate}
\end{prop}
\begin{proof}
  First $p'(\ta_0,Z)$ satisfies $(\ep,\de)$-DP by post-processing. To verify 2), consider
 Let $\alpha\in (0,1)$, and let   \begin{align*}
      P_{N\sim \mathrm{Tulap}}(p'(\ta_0,X+N)\leq \alpha\mid X)
                                    &=P(2\min\{p(\ta_0,X+N),1-p(\ta_0,X+N)\}\leq \alpha\mid X)\\
                                    &=P(p(\ta_0,X+N)\leq \alpha/2\mid X) \\
                                    &\phantom{=}+ P(p(\ta_0,X+N)\geq 1-\alpha/2\mid X)\\
                                    &\phantom{=}- P( 1-\alpha/2\leq p(\ta_0,X+N)\leq \alpha/2\mid X)\\
      &= \phi^*(X) + \psi^*(X)  - 0,
    \end{align*}
    where we use the fact that $1-\alpha/2\geq \alpha/2$, implying that the last probability is zero. To see that $p'(\ta_0,Z)$ is a $p$-value, we compute
    \begin{align*}
      P_{X\sim \ta_0,N}(p'(\ta_0,X+N)\leq \alpha) &= \EE_{X\sim \ta_0} P(p'(\ta_0,X+N)\leq \alpha\mid X)\\
                                                  &=\EE_{X\sim \ta_0}[\phi^*(X) + \psi^*(X)]\\
                                                  &=\alpha/2+\alpha/2.
    \end{align*}
    Finally, to see that $\phi'$ is more powerful than any level $\alpha/2$ test, notice that $\phi^*$ and $\psi^*$ are the most powerful DP tests depending on whether $\ta>\ta_0$ or $\ta<\ta_0$, respectively. Since $\phi' = \phi^* + \psi^*$, it is more powerful than either of these tests. 
\end{proof}

The major benefit of the tests  in Proposition \ref{prop:Bonferroni} is in their simplicity. Generally, they are not optimal in any sense. However, since they are more powerful than any test of level $\alpha/2$, they are not unreasonable tests, and perform relatively well compared to any best case scenario. 


\subsection{DP-UMP unbiased two-sided tests}\label{s:UMPU}
In Section \ref{s:Bonferroni}, we developed two-sided DP tests using a Bonferroni correction. While we were able to show that they are preferred over any level $\alpha/2$ test, they are not optimal when compared to other size $\alpha$ tests.  

In this section, we continue our exploration of DP tests for $H_0: \ta = \ta_0$ versus $H_1: \ta\neq \ta_0$. 
 While one may hope to develop UMP tests in this setting, it is well known that among all tests, there is no UMP test, even without privacy. Indeed, the left-side DP-UMP and the right-side DP-UMP have higher power in different regions. Instead, we must restrict to a smaller class of tests. In classical statistics, it is common to restrict to \emph{unbiased} tests.
We show that there exists a DP-UMP unbiased test (DP-UMPU) for $H_0: \ta=\ta_0$ versus $H_1: \ta\neq \ta_0$, and write the test in terms of the Tulap distribution.

\begin{defn}
  [Unbiased Test]
  A test $\phi: \mscr X^n \rightarrow[0,1]$ is unbiased for $H_0: \ta\in \Ta_0$ versus $H_1: \ta\in \Ta_1$ if for all $\ta_0 \in \Ta_0$ and all $\ta_1\in \Ta_1$ we have that $\EE_{\ta_0} \phi\leq \EE_{\ta_1}\phi.$
\end{defn}
Intuitively, unbiased means that the marginal probability of `Reject' is always higher in the alternative than in the null.

While in the one-sided case, the DP-UMP test increases as much as possible in terms of either $x$ or $-x$, now that we restrict to unbiased tests, the DP-UMP needs to increase as fast as possible in both directions. It turns out that there exists a center $k$, where the DP-UMPU test is symmetric about $k$, and increases as much as possible in both directions, subject to \eqref{DP1}-\eqref{DP4}. This gives the form in Theorem \ref{UMPU}.

\begin{thm}\label{UMPU}
  Let $X\sim \mathrm{Binom}(n,\ta)$, $0<\ta_0<1$ and $0<\al<1$. There exists a UMPU size $\al$ test among $\mscr D_{\ep,\de}^n$ for $H_0: \ta=\ta_0$ versus $H_1: \ta\neq \ta_0$, which is of the form
  \[\phi^*(x) = \begin{cases}
      F_N(x-k-m)&\text{ if } x\geq k\\
      F_N(k-x-m)&\text{ if } x<k,
    \end{cases}\]
  where $k$ and $m$ are chosen such that $\EE_{X\sim \ta_0} (X-n\ta_0) \phi(X)=0$ and $\EE_{X\sim \ta_0} \phi(X) = \al$.
\end{thm}
\begin{proof}
  [Proof Sketch.]
   We must show that there exists $k$ and $m$ which solve the two equations, and then argue that $\phi^*$ is UMP among all level $\al$ tests in $\mscr D_{\ep,\de}^n$. The proof is inspired by the Generalized Neyman Pearson Lemma \citep[Theorem 3.6.1]{Lehmann2008}, and has a similar strategy as Theorem \ref{OneSide1}.

   Let $\ta_1 \neq \ta_0$. We will show that $\phi*$ is most powerful among unbiased size $\al$ tests in $\mscr D_{\ep,\de}^n$ for testing $H_0: \ta=\ta_0$ versus $H_1: \ta=\ta_1$.  Set $f_1(x) = \binom nx \ta_0^x(1-\ta_0)^{n-x}$, $f_2(x) = (x-n\ta_0) \binom nx \ta_0^x(1-\ta_0)^{n-x}$, and $f_3(x) = \binom nx \ta_1^x(1-\ta_1)^{n-x}$. Let $\phi\in \mscr D_{\ep,\de}^n$ be any unbiased size $\alpha$ test, not identical to $\phi^*$. The proof requires verifying the following facts:
   \begin{enumerate}
   \item There exists $k,m\in \RR$ such that $\phi^*$ satisfies  $\EE_{X\sim \ta_0} (X-n\ta_0) \phi^*(X)=0$ and $\EE_{X\sim \ta_0} \phi^*(X) = \al$.
     \item  Since $\phi$ is unbiased, its power must have a local minimum at $\ta_0$ so, $\frac d {d\theta} \EE_\ta \phi\Big|_{\ta=\ta_0}=0$. This is equivalent to requiring that $\sum \phi f_2=0$.
     \item 
       There exists $y_l \leq k\leq y_u$ (integers) such that $\phi^*(x) \geq \phi(x)$ when $x\geq y_u$ or $x\leq y_l$, and $\phi^*(x)\leq \phi(x)$ when $y_l\leq x\leq y_u$.
   \item There exists $k_1,k_2\in \RR$ such that $f_3(x)\geq k_1f_1(x) + k_2f_2(x)$ when $x\not\in (y_l,y_u)$ and $f_3(x)\leq k_1f_1(x) + k_2f_2(x)$ when $x\in (y_l,y_u)$.
   \end{enumerate}
   We have established that $\phi^*(x)\geq \phi(x)$ for $f_3(x)\geq k_1f_1(x)+k_2f_2(x)$ and $\phi^*(x)\leq \phi(x)$ for $f_3(x) \leq k_1f_1(x)+k_2f_2(x)$. Then for all $x\in \{0,1,2,\ldots, n\}$,
      \[(\phi^*(x) - \phi(x))(f_3(x)-k_1f_1(x)-k_2f_2(x))\geq 0.\]
      Summing both sides over $x$ implies
        $\EE_{X\sim \ta_1}\phi^*(X)\geq \EE_{X\sim \ta_1}\phi(X).$
        Note that by taking $\phi(x) =\alpha$, we see that $\phi^*$ is indeed unbiased. 
      Since our argument did not depend on the choice of $\ta_1$ or $\phi$, we conclude that $\phi^*$ is UMP-$\al$ among unbiased tests in $\mscr D_{\ep,\de}^n$. 
\end{proof}

  In Theorem \ref{OneSide1}, we were able to derive DP $p$-values that agree with these tests. However,
   in Theorem \ref{UMPU}, the quantity $k$ depends on $n$, $\alpha$, and $\ta_0$, and there is no clear functional form of $k$ in terms of these quantities. Thus it does not seem that there is a simple formula for the $p$-values of the test in Theorem \ref{UMPU}. The following corollary shows that in the case when $\ta_0 = \frac 12$ and $k= \frac n2$, a convenient form for the $p$-value does exist. 

\begin{cor}\label{UMPUhalf}
  In the setup of Theorem \ref{UMPU}, if $\ta_0 = \frac 12$ then $k = \frac n2$. Let $Z\mid X\sim \mathrm{Tulap}\l(X,b=e^{-\ep},\frac{2\de b}{1-b+2\de b}\r)$, then the corresponding $p$-value is
  \[p(Z) = P_{X\sim \ta_0,N}\left(|X+N-n/2|\geq |Z-n/2|\Big| Z\right),\] which can be computed via Algorithm \ref{NearlyUnbiasedPvalue}, setting $\ta_0=\frac 12$. 
\end{cor}
\begin{proof}[Proof Sketch.]
  It suffices to check that when $k=\frac n2$, the test is unbiased. This is done using the symmetry of both $\phi(x)$ and $f_X(x)$ about $\frac n2$. 
\end{proof}
\begin{remark}
While Corollary \ref{UMPUhalf} only applies in the case that $\ta_0 = \frac 12$, this is in fact a common setting. This arises when we are interested in testing whether two mutually exclusive (and collectively exhaustive) events are equally likely, such as whether the probability of being born male versus female is $\frac 12$. In Section \ref{DistributionFree}, we see that for the sign and median test, when testing whether the medians of two random variables are equal or not, this can be expressed as testing $H_0: \ta = \frac 12$ versus $H_1: \ta\neq \frac 12$.
\end{remark}

\subsection{Asymptotically unbiased two sided tests}\label{s:ApproxUnbiased}
In the previous section, we developed the DP-UMPU two-sided test, and showed that only in the case where $\ta_0 = 1/2$ we can easily compute $p$-values. Otherwise, since $k$ depends on $\alpha$ and there is no natural test statistic, the problem is more challenging. 
Nevertheless, based on Corollary \ref{UMPUhalf} we conjecture that $|X+N-n\ta_0|$ is a test statistic which provides a close approximation to the DP-UMPU test. In Section \ref{s:Simulations2}, we see that this asymptotically unbiased test performs very similarly to the UMPU test from Section \ref{s:UMPU} for finite samples. 

\begin{algorithm}
  \caption{Asymptotically unbiased DP $p$-value}
  \scriptsize
  INPUT: $n\in \NN$, $\ta_0 \in (0,1)$, $\ep>0$, $\de\geq 0$, $Z\sim \mathrm{Tulap}\l(X,b=e^{-\ep},q = \frac{2\de b}{1-b+2\de b}\r)$
  \begin{algorithmic}[1]
    \setlength\itemsep{0em}
    \STATE Set $T = |Z-n\ta_0|$
    \STATE Call $p(\ta, Z)$ the $p$-value computed by Algorithm \ref{PValueAlgorithm}.
    \STATE Set $p = p(\ta_0, T+n\ta_0) + 1-p(\ta_0, n\ta_0-T)$
  \end{algorithmic}
  OUTPUT: $p$
  \label{NearlyUnbiasedPvalue}
\end{algorithm}

\begin{prop}\label{prop:Nearly}
  The output of Algorithm \ref{NearlyUnbiasedPvalue} is
  \[p(\ta_0,Z) = P_{X\sim \ta_0,N}\left(|X+N-n\ta_0|>|Z-n\ta_0|\Big|Z\right),\] which is a $p$-value for $H_0: \ta=\ta_0$ versus $H_0: \ta\neq \ta_0$ and satisfies $(\ep,\de)$-DP. The corresponding test $\phi(x) = P_N(p(\ta_0,Z)\leq \alpha\mid X)$ is of the form of Theorem \ref{UMPU}, with $k = n\ta_0$. 
\end{prop}
\begin{proof}[Proof Sketch.]
 It is easy to verify that when $\ta = \ta_0$, $p(\ta_0, Z)$ is marginally distributed as $U(0,1)$.  Since $p(\ta_0, Z) \sim U(0,1)$, we have that $P_{\ta_0}(p(\ta_0, Z)\geq \al) = \al.$ 
 The $p$-value satisfies $(\ep,\de)$-DP since it is a post-processing of $Z$.
\end{proof}

\begin{prop}\label{prop:Unbiased}
  In the setting of Theorem \ref{UMPU}, holding $\ep$, $\de$, $\al$, and $\theta_0$ all fixed, the test in Proposition \ref{prop:Nearly} is asymptotically unbiased.
\end{prop}
\begin{proof}[Proof Sketch.]
  In the full proof of Theorem \ref{UMPU}, we saw that if $\phi$ is of the form in Theorem  \ref{UMPU} and  $\EE_{\ta_0} (X-n\ta_0)\phi(X)=0$, then $\phi$ is unbiased. Let $\phi$ be the test in Proposition \ref{prop:Nearly}. Then it suffices to show that $\lim_{n\rightarrow \infty} \EE_{\ta_0}\frac{X-n\ta_0}{\sqrt {n\ta_0(1-\ta_0) }}\frac{\phi(X)}{\sqrt n} = 0$. Recall that if $X\sim \mathrm{Binom}(n,\ta_0)$
  \[\frac{X-n\ta_0}{\sqrt{n\ta_0(1-\ta_0)}} \overset d\rightarrow N(0,1).\]
  Using the fact that $\phi(x)$ is symmetric about $k=n\ta_0$, we see that the expectation is the integration of the product of two even functions and one odd function. Hence the expectation is zero. 
  \end{proof}


\begin{remark}
  Since Proposition \ref{prop:Unbiased} shows that the test $\phi'$ in Proposition \ref{prop:Nearly} is asymptotically unbiased, and since it is of the form of the UMPU test $\phi^*$ of Theorem \ref{UMPU}, as the sample size increases, the power of the test $\phi'$ is very similar to that of $\phi^*$. In Section \ref{s:Simulations2}, we see that even at $n=30$, the performance is very close between $\phi'$ and $\phi^*$. 
\end{remark}


\section{Confidence intervals}\label{s:CI}
\subsection{Background and notation}
A confidence set is a popular method of expressing uncertainty about a population quantity. Since all estimates have some error in them, a confidence set communicates the set of values in which we expect the population quantity to lie. While confidence sets can be of arbitrary forms, typically we prefer confidence sets which are intervals, since this simpler form improves interpretability.

\begin{defn}
  [Confidence Interval]
  Let $X_i \iid f_\ta$, where $\ta\in \Ta\subset \RR$. A \emph{(random) confidence interval (CI)} is a set of random variables $\mscr C = \{C(\ul x) \mid \ul x\in \mscr X^n\}$, each of which takes values in $\{[a,b]\in \RR^2 \mid a\leq b\}$. We say that $\mscr C$ has coverage $\ga$ if for all $\ta\in \Ta$,
  \[P_{\ul X \sim \ta}( \ta \in C(\ul X))\geq \ga.\]
  If one of $a$ or $b$ in a confidence interval $\mscr C$ is constant, then we call $\mscr C$ a \emph{one-sided confidence interval}, otherwise we call $\mscr C$ a \emph{two-sided confidence interval}. 
\end{defn}

For convenience, we will often suppress the dependence of a confidence interval on $x$, and simply write $C$ rather than $C(x)$. In classical statistics, there is a well known connection between hypothesis tests and confidence sets (see  \citealp[Chapter 9]{Casella2002}). For a deterministic test of level $\al$, with rejection region $R$, the set $\Ta\setminus R$ is a confidence set with coverage $1-\al$. For randomized tests, it is more convenient to work with $p$-values. The following Proposition shows how one can use a $p$-value to build a confidence set.

\begin{prop}\label{p-CI}
If $p(\ta\mid X)$ is a $p$-value, then $C(X) = \{\ta\mid p(\ta\mid X)\geq \al\}$ is a confidence set with coverage $1-\al$. 
\end{prop}

In order to decide whether one confidence interval is to be preferred over another, we require some criteria. In classical statistics, one considers {\em uniformly most accurate} (UMA) confidence intervals, which have properties related to UMP tests. 
UMA confidence intervals are defined in terms of \emph{false coverage}, which is the analogue of the power of the corresponding test. The UMA  property is important in theory and practice,  because it results in smaller confidence intervals, and more accurately  communicates the uncertainty of the parameter in question. Our definitions of false coverage and UMA follow that of \citet[Section 9.3.2]{Casella2002}.


\begin{defn}
  [False Coverage and Uniformly Most Accurate]
  Let $C$ be a confidence interval for $\ta\in \Ta$. The probability of \emph{false coverage} is a function of two values $\ta_0$, $\ta_1$:
  \begin{align*}
    P_{\ta_1}(\ta_0 \in C)\quad \text{for $\ta_0\neq\ta_1$}&\quad \text{if $C(X) = [L(X),U(X)]$},\\
    P_{\ta_1}(\ta_0 \in C)\quad \text{for $\ta_0<\ta_1$}&\quad \text{if $C(X) = [L(X),\max\{\Ta\}]$},\\
    P_{\ta_1}(\ta_0 \in C)\quad \text{for $\ta_0>\ta_1$}&\quad \text{if $C(X) = [\min\{\Ta\},U(X)]$}.
  \end{align*}
  The \emph{Uniformly Most Accurate (UMA)} confidence interval among a set of confidence intervals, minimizes the false coverage for all valid $\ta_0$ and $\ta_1$.
\end{defn}

  \subsection{One-sided confidence intervals}\label{s:OneSide}
  In this section, we show how we can use our DP-UMP tests to produce private confidence intervals for Bernoulli data. In the following Theorem, we show that the one-sided confidence interval based on our DP-UMP one-sided test is UMA. Furthermore, this interval is still a function of our private test statistic $Z = X+N$, and so after releasing $Z$, there is no additional cost to privacy when providing this confidence interval. 

\begin{thm}\label{OneSideCI}
  Let $Z = X+N$, where $X\sim \mathrm{Binom}(n,\ta)$ and $N\sim \mathrm{Tulap}(0,b=e^{-\ep},q=\frac{2\de b}{1-b+2\de b})$, and let $p(\ta_0,Z)$ be the one-sided private $p$-value for $H_0: \ta\leq \ta_0$ versus $H_1: \ta\geq \ta_0$, defined in Theorem \ref{PValueResult}. Let $\al\in (0,1)$ be given. The confidence interval $C_\al^* = \{\ta_0 \mid p(\ta_0,Z)\geq \al\}$ 
  is the UMA $(\ep,\de)$-DP confidence interval of the form $[L,1]$ with coverage $1-\al$,
\end{thm}

\begin{proof}
  Releasing $C^*_{\al}$ satisfies $(\ep,\de)$-DP by the post-processing property of DP. The object $C^*_{\al}$ is of the form $[L,1]$ by the monotonicity of $p(\ta_0,Z)$ in $\ta_0$ (for a fixed $Z$). The coverage of $C^*_{\al}$ is $1-\al$, by the fact that $p(\ta_0,Z)$ is a $p$-value.

  Next we check that $C^*_{\al}$ is in fact UMA. Suppose to the contrary that there exists another DP test $C'_\al$ with coverage $1-\al$, and there exists $\ta_0<\ta_1$ such that 
  $\displaystyle P_{\ta_1}(\ta_0 \in C'_\al)< P_{\ta_1}(\ta_0 \in C^*_\al),$
  which is equivalent to
  \begin{equation}\label{UMA1}
    P_{\ta_1}(\ta_0 \not\in C'_\al) > P_{\ta_1}(\ta_0 \not\in C^*_\al).
  \end{equation}

  Notice that $\phi^*(x) = P(\ta_0 \not \in C^*_\alpha\mid x)$ is the UMP size $\alpha$ test from Theorem \ref{UMP2} for $H_0: \ta\leq \ta_0$ versus $H_1: \ta>\ta_0$, and $\phi'(x) = P(\ta_0 \not \in C'_\alpha)$ is another test, which is also level $\alpha$ for the same test. Observe that $\phi'$ satisfies $(\ep,\de)$-DP since it outputs `Reject' if and only if $I(\ta_0 \not \in C'_\alpha)=1$, which is a post-processing of the DP confidence interval $C'_\alpha$.

  Now, note that \eqref{UMA1} can be equivalently expressed as $\EE_{\ta_1} \phi'>\EE_{\ta_1} \phi^*$, which says that $\phi'$ has more power at $\ta_1$ than $\phi^*$, which contradicts that $\phi^*$ is UMP-$\al$  among $\mscr D_{\ep,\de}^n$. We conclude that $C^*_\al$ is UMA.

\end{proof}

\begin{cor}
  Using the same setup as Theorem \ref{OneSideCI}, 
  the interval $C_\al =\{\ta_0 \mid (1-p(\ta_0,Z))\geq \al\}$ is the UMA $(\ep,\de)$-DP confidence interval of the form $[0,U]$, with coverage $1-\al$.
\end{cor}

\begin{remark}
  The value $L^*$ in the interval $C^*_{\al} = [L^*,1]$ of Theorem \ref{OneSideCI} can be easily computed by minimizing $(p(\ta_0,Z)-\al)^2$ over the interval $\ta_0\in [0,1]$. This can be done using standard optimization software. 
\end{remark}

\subsection{Two-sided confidence intervals}\label{s:TwoSide}
As we saw in Theorem \ref{OneSideCI}, when $p(\ta_0,Z)$ is a one-sided $p$-value, the set $\{\ta_0 \mid p(\ta_0,Z)\geq \alpha\}$ forms a one-sided confidence interval. Similarly, if $p(\ta_0,Z)$ is a two-sided $p$-value, then $\{\ta_0\mid p(\ta_0,Z)\geq \alpha\}$ is of the form $[L,U]$. In this section, we will consider the DP confidence intervals produced by each of our proposed two-sided tests from Sections \ref{s:Bonferroni}-\ref{s:ApproxUnbiased}.

First, we introduce our optimality criterion for two-sided confidence intervals. Just as there is generally no UMP two-sided test, there does not exist a UMA two-sided confidence interval. In Section \ref{s:UMPU}, we saw that for two-sided tests, we imposed the condition of unbiasedness in order to obtain a UMP test. Similarly, we will consider an analogous notion of \emph{unbiasedness} for confidence intervals. Intuitively, a confidence interval is unbiased if the probability of false coverage is always smaller than the true coverage. Unbiased confidence intervals and unbiased hypothesis tests are in one-to-one correspondence via the connection in Proposition \ref{p-CI}.

\begin{defn}
  [Unbiased Confidence Interval]
  Let $C(\ul X)$ be a confidence interval for $\ta$. We call $C$ \emph{unbiased} if for all $\ta\neq \ta'$, $P_{\ta}(\ta' \in C)\leq P_{ \ta}(\ta\in C)$. 
\end{defn}

The next result shows that our DP-UMPU test from Theorem \ref{UMPU} leads to a DP-UMA unbiased (DP-UMAU) confidence interval.

\begin{prop}\label{UMAU}
  Let $\ep>0$, $\de\geq 0$, $X\sim \mathrm{Binom}(n,\ta)$ where $\ta$ is unknown. Let $Z\mid X\sim \mathrm{Tulap}\left(X,b=e^{-\ep},q=\frac{2\de b}{1-b+2\de b}\right)$.  We construct the corresponding randomized $p$-value as
  \[p^*(Z) = \min \left\{\alpha \Big| |Z-k(\alpha)|\geq m(\alpha)\right\},\]
  where $k(\cdot)$ and $m(\cdot)$ satisfy the requirements of Theorem \ref{UMPU}. 
  The set $C^* = \{\ta\mid p^*(x,U)\geq \alpha\}$ is an unbiased, DP confidence interval with coverage $(1-\alpha)$, and $C^*$ is the DP-UMAU confidence interval with coverage $(1-\alpha)$.
\end{prop}
The proof of Proposition \ref{UMAU} is similar to the proof of Theorem \ref{OneSideCI}, and is postponed to Section \ref{s:Appendix}.

While Proposition \ref{UMAU} gives the DP-UMAU confidence interval, it is not easy to implement, since $k$ and $m$ do not have simple closed forms, as discussed in Section \ref{s:UMPU}. Instead, we can use Proposition \ref{prop:CI} to produce computationally convenient confidence intervals based on the $p$-values from Proposition \ref{prop:Bonferroni} and Algorithm \ref{NearlyUnbiasedPvalue}.

\begin{prop}\label{prop:CI}
   Let $\ep>0$, $\de\geq 0$, $X\sim \mathrm{Binom}(n,\ta)$ where $\ta$ is unknown, and $Z\mid X\sim \mathrm{Tulap}\left(X,b=e^{-\ep},q=\frac{2\de b}{1-b+2\de b}\right)$. Consider the two quantities
  \begin{enumerate}
  \item $C^1_\al = \{\theta_0 \mid p'(\ta_0, Z)\geq \alpha\}=(C^*_{\al/2})\setminus (C^*_{1-\al/2})$, where $p'(\ta, Z)$ is the Bonferroni $p$-value from Proposition \ref{prop:Bonferroni} and $C^*_{\al}$ is the one-sided confidence interval from Theorem \ref{OneSideCI}.
    \item  $C^2_\al = \{\theta_0\mid p(\ta_0, Z)\geq \al\}$, where $p(\ta_0,Z)$ is the $p$-value from Proposition \ref{prop:Nearly}.
    \end{enumerate}
    Both $C^1_\al$ and $C^2_\al$ are $(\ep,\de)$-DP confidence intervals of the form $[L,U]$ with coverage $(1-\alpha)$.
  \end{prop}

  \begin{remark}
    Since the confidence interval $C^2_\al$ from Proposition \ref{prop:CI} is based on the approximation to the DP-UMPU test, it serves as an approximation to the DP-UMAU confidence interval from Proposition \ref{UMAU}. 
  \end{remark}

  Similar to Proposition \ref{prop:Bonferroni}, which stated that the Bonferroni two-sided test is uniformly more powerful than any level $\alpha/2$ DP test, the following Corollary shows that $C^1_\al$ in Proposition \ref{prop:CI} is uniformly more accurate than any DP confidence interval with coverage $1-\alpha/2$. The proof is found in Section \ref{s:Appendix}.

  \begin{cor}\label{BonferroniCI}
    In the setting of Proposition \ref{prop:CI}, $C^1_\al$ is uniformly more powerful than any $(\ep,\de)$-DP confidence interval with coverage $1-\alpha/2$.
  \end{cor}





\section{Confidence distributions and distribution-free inference} 
\subsection{Confidence distributions}\label{s:CD}
A confidence distribution is a frequentist estimator, which contains information to produce hypothesis tests, confidence intervals, $p$-values, point estimates, etc (see \citealp{Xie2013} for an introduction to Confidence Distributions). Much like in Bayesian statistics, the posterior distribution is used to do inference, a confidence distribution contains the relevant information for frequentist statistics. Intuitively, a confidence distribution $\mu$ is a probability measure on $\Ta$ such that for $S\subset \Ta$, $\mu(S)$ is the coverage of $S$. Confidence distributions also have the property that the cdf of $\mu$ evaluated at $\ta_0$ is a $p$-value for $H_0: \ta\leq \ta_0$ versus $H_1: \ta>\ta_0$.

The goal of this section is to release a confidence distribution, which satisfies DP. In particular, we show that using our one-sided DP-UMP tests we can produce optimal DP confidence distributions.

\begin{defn}
  [Confidence Distribution: \citet{Xie2013}]
  Let $X_i \iid f_\ta$ for $\ta\in \Ta$ and $X_i \in \mscr X$. A \emph{confidence distribution} is a family of random variables $\{H_n(\ul x,\ta)\mid \ul x\in \mscr X^n,\ta\in \Ta\}$ (we will suppress the dependence on $\ul x$ and write $H_n(\ta)$), each of which takes values in $[0,1]$ such that 
  \begin{enumerate}
  \item for each $\ul x \in \mscr X^n$, $H_n(\cdot)$ is a cdf on $\Ta$, and
  \item at the true value $\ta = \ta_0$, $H_n(\ta_0) = H_n(X,\ta_0)\sim U[0,1]$ (over randomness of $H_n$ and over $X$).
  \end{enumerate}
\end{defn}

Supposing that we have two methods of constructing confidence distributions, what criteria should we use to choose between them?  In the following definition, we say that one confidence distribution is superior to another if the mass is more closely distributed near the true value $\ta_0$.

\begin{defn}
   For real-valued random variables $X,Y$, $X\overset {sto}\leq Y$ means that $P(X\leq t)\geq P(Y\leq t)$ for all $t\in \RR$.   Let $H_1$ and $H_2$ be two confidence distributions. We say that $H_1$ is \emph{superior} to $H_2$ at $\ta = \ta_0$ if for all $\ep>0$, $H_1(\ta_0-\ep) \overset {sto}\leq H_2(\ta_0-\ep)$ and $1-H_1(\ta_0 +\ep) \overset {sto}\leq 1-H_2(\ta_0+\ep)$.
 \end{defn}

 In Section 5 of \citet{Xie2013}, they discuss how using a UMP one-sided test results in the optimal confidence distribution. Theorem \ref{thm:CD} below similarly shows that our DP-UMP one-sided test results in the optimal DP confidence distribution.

\begin{thm}\label{thm:CD}
  Let $Z = X+N$, where $X\sim \mathrm{Binom}(n,\ta)$ and $N\sim \mathrm{Tulap}(0,b=e^{-\ep},q=\frac{2\de b}{1-b+2\de b})$, and let $p(\ta_0,Z)$ be the one-sided private $p$-value for $H_0: \ta\leq \ta_0$ versus $H_1: \ta\geq \ta_0$. Define $H^*_n(\ta_0) = p(\ta_0,Z)$. Then $H^*_n$ is a confidence distribution which satisfies $(\ep,\de)$-DP, and is superior to any other $(\ep,\de)$-DP confidence distribution.
\end{thm}
\begin{proof}
  That $H^*_n$ satisfies $(\ep,\de)$-DP follows by the post-processing property of DP. The fact that $H^*_{n}$ is a confidence distribution follows from the fact that $p(\ta_0,Z)$ is monotonic in $\ta_0$, and $p(\ta_0,Z)\in [0,1]$. If $H^*_n$ were not superior, then this contradicts that $p(\ta_0,Z)$ corresponds to the UMP test among $\mscr D_{\ep,\de}^n$ for $H_0: \ta\leq \ta_0$ versus $H_1: \ta>\ta_0$. 
\end{proof}

\subsection{Application to distribution-free inference}\label{DistributionFree}
In this section, we show how our DP-UMP tests for count data can be used to test certain hypotheses for continuous data. In particular, we give a DP version of the sign and median test allowing one to test the median of either paired or independent samples. For an introduction to the sign and median tests, see Sections 5.4 and 6.4 of \citet[]{Gibbons2014}.
Let $\ep>0$ and $\de\in [0,1)$ be given, and let $N\sim \mathrm{Tulap}(0,b,q)$ for $b=e^{-\ep}$ and $q = \frac{2\de b}{1-b-2\de b}$.

{\bfseries Sign test:}
We observe $n$ iid pairs $(X_i,Y_i)$ for $i=1,\ldots, n$. Then for all $i=1,\ldots,n$, $X_i \overset d= X$ and $Y_i \overset d =Y$ for some random variables $X$ and $Y$. We assume that for any pair $(X_i,Y_i)$ we can determine if $X_i>Y_i$ or not. For simplicity, we also assume that there are no pairs with $X_i=Y_i$. Denote the unknown probability $\ta = P(X>Y)$. We want to test a hypothesis such as $H_0: \ta\leq\ta_0$ versus $H_1: \ta>\ta_0$. The sign test uses the test statistic $T = \#\{X_i>Y_i\}$. Since the sensitivity of $T$ is $1$, by Theorem \ref{TulapDP}, $T+N$ satisfies $(\ep,\de)$-DP. Note that the test statistic is distributed as $T \sim \mathrm{Binom}(n,\ta)$.  Using Algorithm \ref{PValueAlgorithm}, we  obtain a private $p$-value for the sign test as a post-processing of $T+N$.

To test whether $ \mathrm{median}(X) =  \mathrm{median}(Y)$, we consider the hypothesis $H_0: \ta=\frac 12$ versus $H_1: \ta\neq \frac 12$. Using the same test statistic $Z = T+N$, we obtain a $p$-value for the sign test via Algorithm \ref{NearlyUnbiasedPvalue}.

{\bfseries Median test:}
We observe two independent sets of iid data $\{X_i\}_{i=1}^n$ and $\{Y_i\}_{i=1}^n$, where all $X_i$ and $Y_i$ are distinct values, and we have a total ordering on these values. We assume that there exists random variables $X$ and $Y$ such that  $X_i\overset d =X$ and $Y_i \overset d = Y$ for all $i$. We want to test $H_0: \mathrm{median}(X)\leq \mathrm{median}(Y)$ versus $H_1: \mathrm{median}(X)>\mathrm{median}(Y)$. The median test uses the test statistic $T = \#\{i \mid \mathrm{rank}(X_i) > n\}$, where $\mathrm{rank}(X_i) = \#\{X_j\leq X_i\}+\#\{Y_j\leq X_i\}$. Since the sensitivity of $T$ is $1$, by Theorem \ref{TulapDP}, $T+N$ satisfies $(\ep,\de)$-DP. When $\mathrm{median}(X) = \mathrm{median}(Y)$, $T\sim \mathrm{HyperGeom}(n=n,m=n,k=n)$. Using Algorithm \ref{PValueAlgorithm}, with $\ul B$ replaced with the pmf of $\mathrm{HyperGeom}(n=n,m=n,k=n)$, we  obtain a private $p$-value for the median test as a post-processing of $T+N$.

To test whether $ \mathrm{median}(X) =  \mathrm{median}(Y)$, we consider the hypothesis $H_0: \ta=\frac 12$ versus $H_1: \ta\neq \frac 12$. Using the same test statistic $Z = T+N$, we obtain a $p$-value for the sign test via Algorithm \ref{NearlyUnbiasedPvalue}, with $\ul B$ replaced with the pmf of $\mathrm{HyperGeom}(n=n,m=n,k=n)$.

\section{Simulations}\label{s:Simulations}
\subsection{One-sided hypothesis testing simulations}\label{s:Simulations1}
\begin{figure}
\begin{minipage}{.48\linewidth}
  \includegraphics[width = \linewidth]{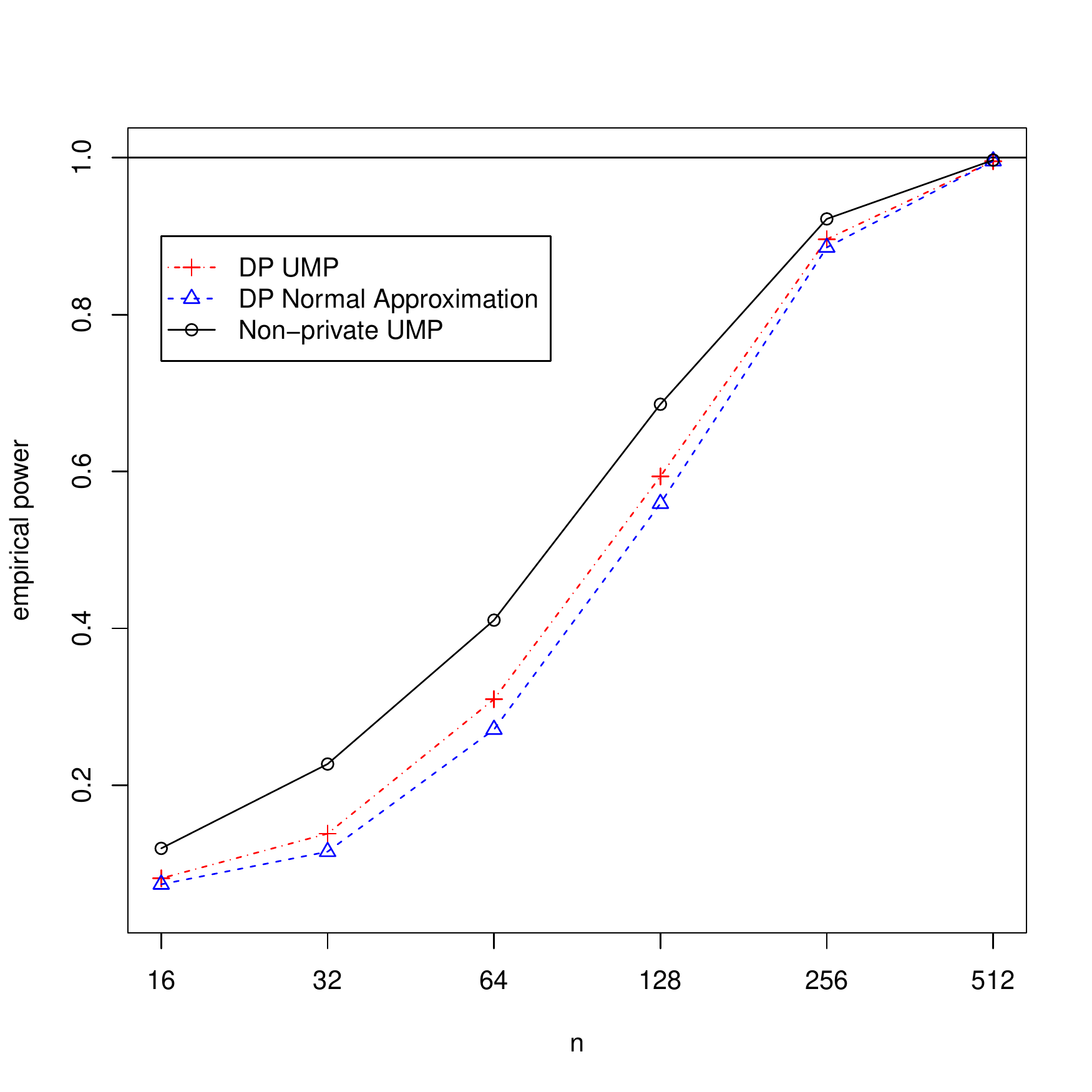}
  \captionof{figure}{Empirical power for UMP and Normal Approximation tests for $H_0: \theta\leq.9$ versus $H_1: \theta\geq .9$. The true value is $\theta = .95$. $\ep=1$ and $\de=0$. $n$ varies along the $x$-axis.}
  \label{fig:Power}
\end{minipage}
\hspace{.02\linewidth}
\begin{minipage}{.48\linewidth}
  \includegraphics[width = \linewidth]{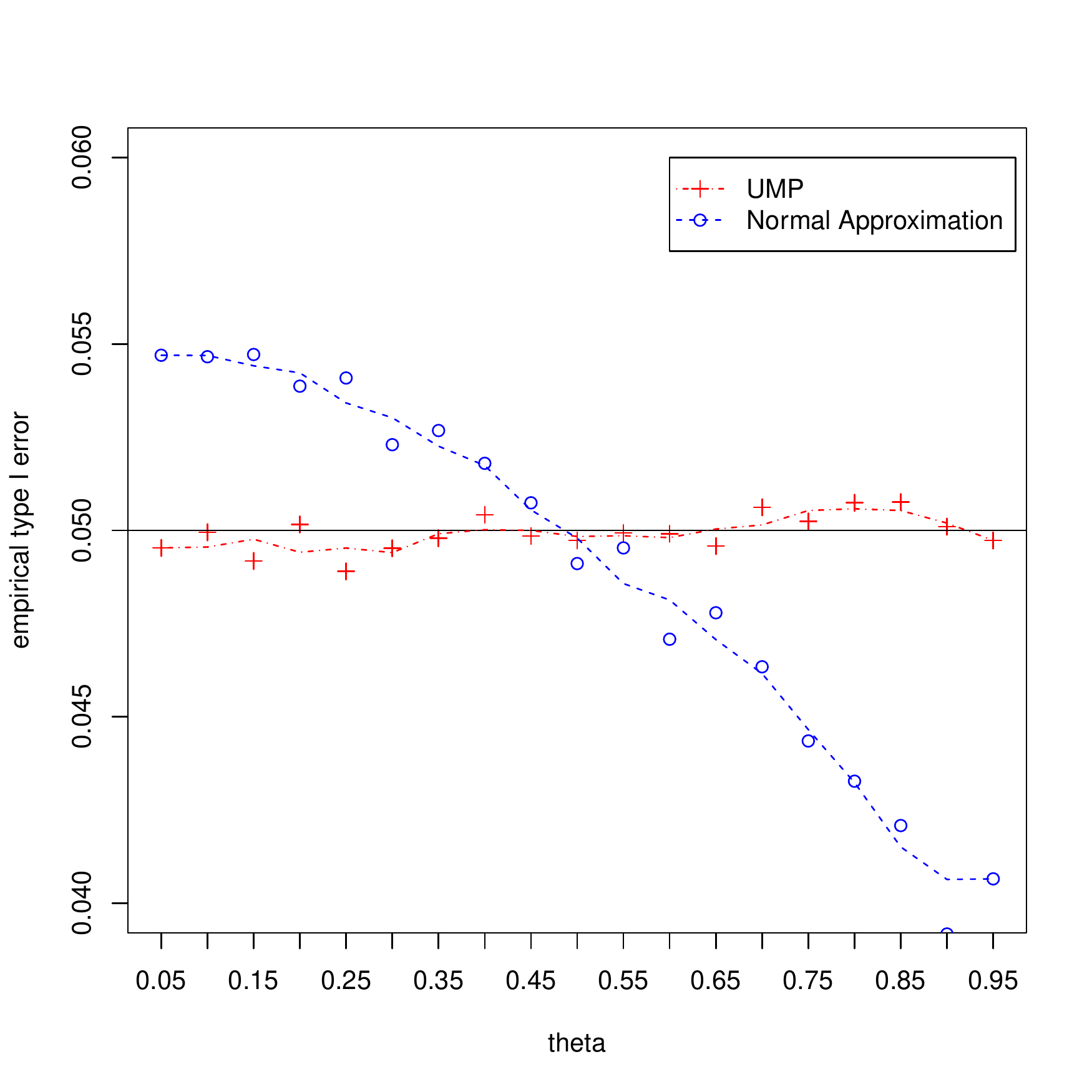}
  \captionof{figure}{Empirical type I error $\al$ for UMP and Normal Approximation tests for $H_0: \theta \leq \ta_0$ versus $H_1: \theta \geq \theta_0$.  $x$-axis is $\ta_0$. $n=30$, $\ep=1$, and $\de=0$. Target is  $\al=.05$.}
  \label{fig:TypeIerror}
\end{minipage}
\end{figure}

 In this section, we study both the empirical power and the empirical type I error of our DP-UMP test against the normal approximation proposed by \citet{Vu2009}. We define the empirical power to be the proportion of times a test `Rejects' when the alternative is true, and the empirical type I error as the proportion of times a test `Rejects' when the null is true. For our simulations, we focus on small samples as the noise introduced by DP methods is most impactful in this setting.

In Figure \ref{fig:Power}, we plot the empirical power of our UMP test, the Normal Approximation from \citet{Vu2009}, and the non-private UMP. For each $n$, we generate 10,000 samples from $\mathrm{Binom}(n,.95)$. We  privatize each $X$ by adding $N\sim\mathrm{Tulap}(0,e^{-\ep},0)$ for the DP-UMP and $L\sim \mathrm{Lap}(1/\ep)$ for the Normal Approximation. We compute the UMP $p$-value via Algorithm \ref{PValueAlgorithm} and the approximate $p$-value for $X+L$, using the cdf of $N\l(X, n/4+2/\ep^2\r)$. The empirical power is given by $(10000)^{-1}\#\{\text{$p$-value$<.05$}\}$. The DP-UMP test  indeed gives higher power compared to the Normal Approximation, but the approximation does not lose too much power. Next we see that type I error is another issue.

In Figure \ref{fig:TypeIerror} we plot the empirical type I error of the DP-UMP and the Normal Approximation tests. We fix $\ep=1$ and $\de=0$, and vary $\ta_0$. For each $\ta_0$, we generate 100,000 samples from $\mathrm{Binom}(30,\ta_0)$. For each sample, we compute the DP-UMP and Normal Approximation tests at type I error $\al=.05$. We plot the proportion of times we reject the null as well as moving average curves. The DP-UMP, which is provably at type I error $\al=.05$ achieves type I error very close to $.05$, but the Normal Approximation has a higher type I error for small values of $\ta_0$, and a lower type I error for large values of $\ta_0$. 

\subsection{Two-sided hypothesis testing simulations}\label{s:Simulations2}
\begin{figure}
\begin{minipage}{.48\linewidth}
    \includegraphics[width=\linewidth]{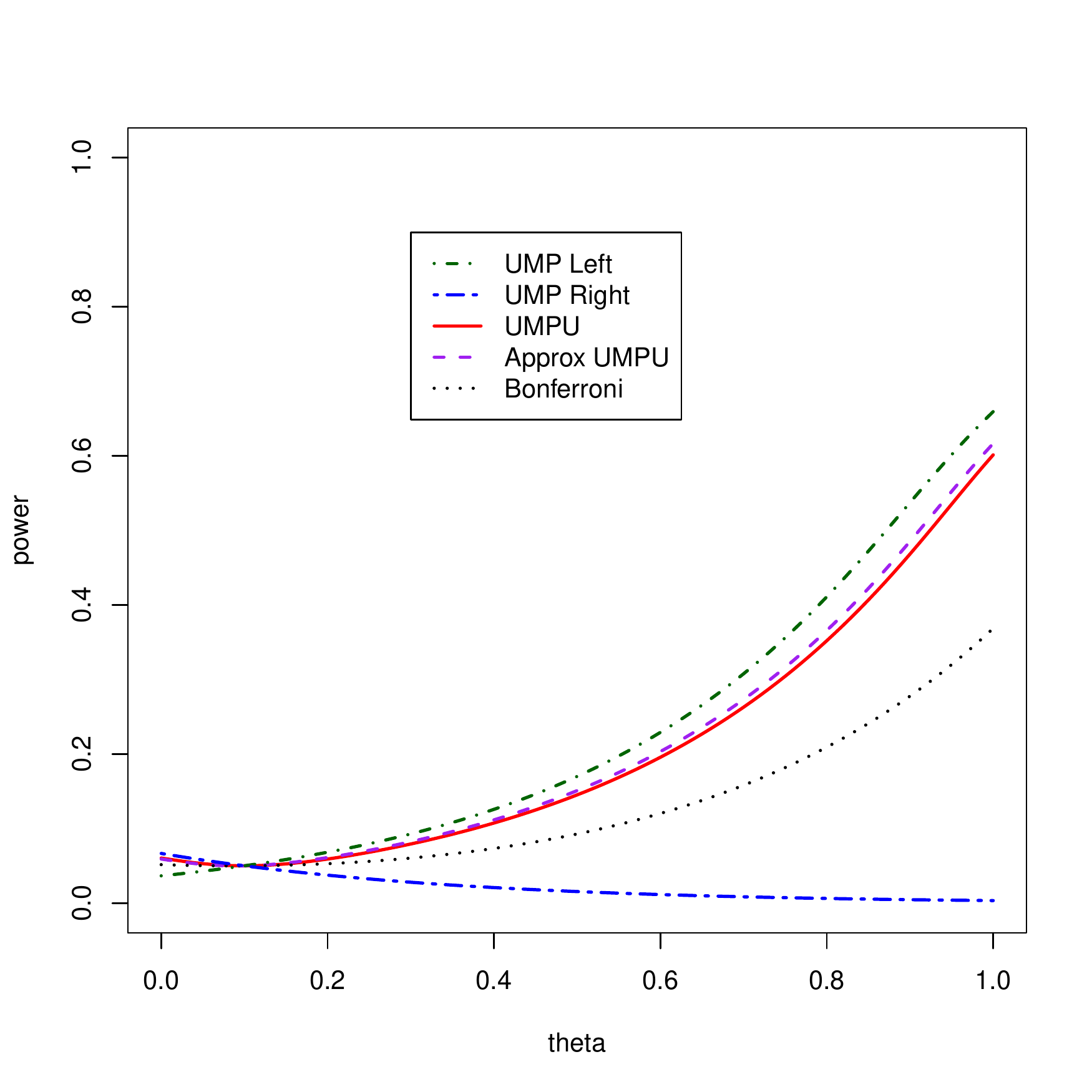}
    \captionof{figure}{$n=30$, $\ta_0=.1$, $\ep=.1$, $\de=0$, $\alpha=.05$. Testing $H_0: \ta=\ta_0$ versus $H_1:\ta_0\neq \ta_0$. $x$-axis is the truth.}
\label{fig:TwoSideTest_1.pdf}
\end{minipage}
\hspace{.02\linewidth}
\begin{minipage}{.48\linewidth}
    \includegraphics[width=\linewidth]{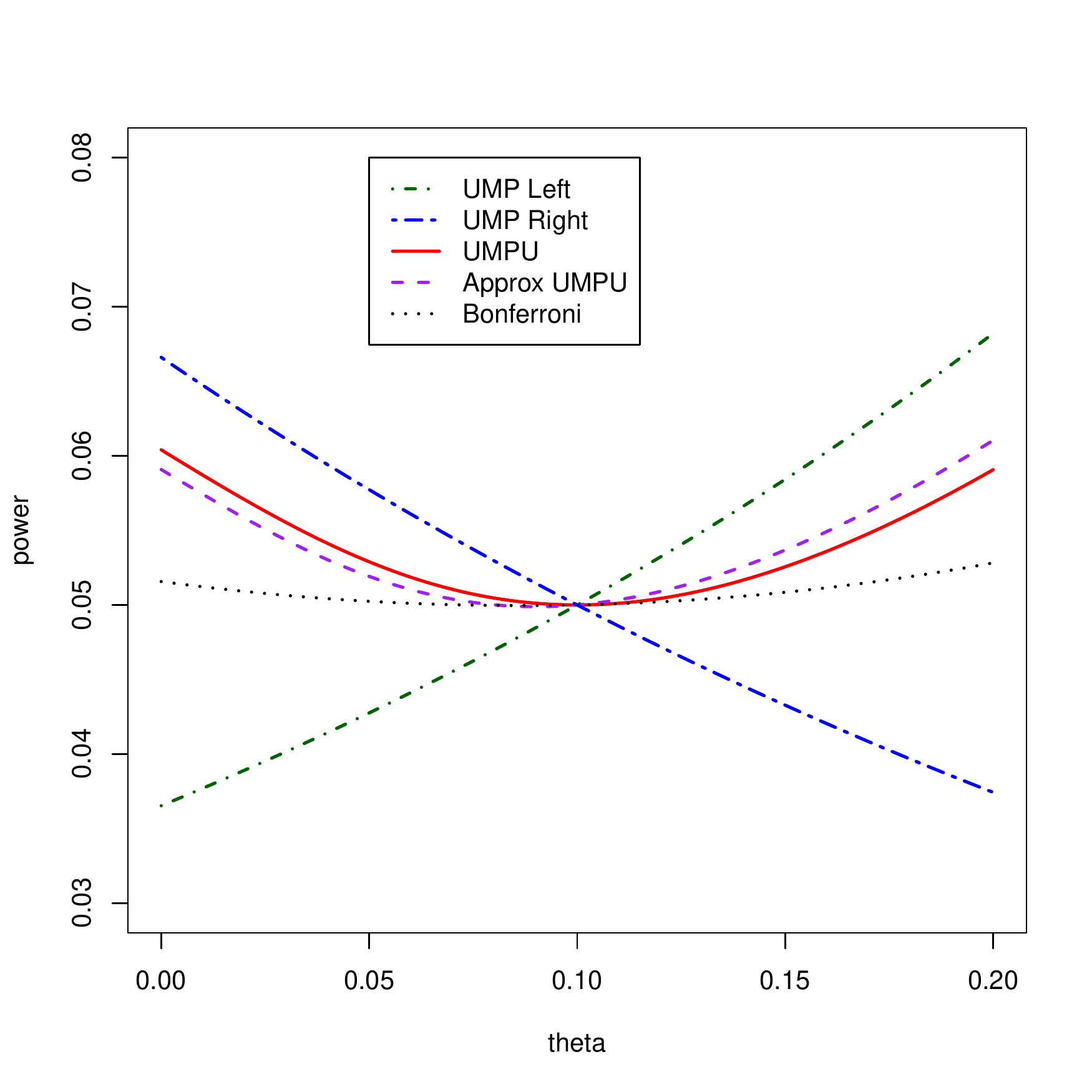}
    \captionof{figure}{Same parameters as Figure \ref{fig:TwoSideTest_1.pdf}. Testing $H_0: \ta=\ta_0$ versus $H_1:\ta_0\neq \ta_0$. $x$-axis is the truth for $0\leq \ta\leq .2$.}
\label{fig:TwoSideTest_12.pdf}
\end{minipage}
\end{figure}

In this section, we compare the various tests we have developed for $H_0: \ta = \ta_0$ versus $H_1: \ta \neq \ta_0$. We also consider the one-sided tests, since these provide upper bounds for the power of the two-sided tests. For each of the simulations, we were able to compute the power exactly, since we have closed forms for the tests in terms of the Tulap distribution, and power is just an expected value.

In Figures \ref{fig:TwoSideTest_1.pdf}-\ref{fig:TwoSideN.pdf}, ``UMP Left'' corresponds to the DP-UMP test for $H_0: \ta\leq \ta_0$, ``UMP Right'' corresponds to the DP-UMP test for $H_0: \ta\geq \ta_0$, ``UMPU'' corresponds to the test from Theorem \ref{UMPU}, ``Approx UMPU'' corresponds to the test from Section \ref{s:ApproxUnbiased}, and ``Bonferroni'' corresponds to the test from Proposition \ref{prop:Bonferroni}.

In Figures \ref{fig:TwoSideTest_1.pdf} and \ref{fig:TwoSideTest_12.pdf}, we see how the tests perform when the null value is more extreme ($\ta_0 = .1$). As our theory showed, the DP-UMP test for $H_0: \ta\leq \ta_0$ is the most powerful for true values $>\ta_0$, and the DP-UMP test for $H_0: \ta\geq \ta_0$ is the most powerful for true values $<\ta_0$. We see that the DP-UMPU test and the approximately unbiased test perform well on both sides. However, the Bonferroni test suffers a loss in power, demonstrating that either the DP-UMPU, or the approximately unbiased test should be preferred.

In Figure \ref{fig:TwoSideTest_5.pdf} we study our tests when $n=100$ and $\ta_0=.5$. In this case, the approximate test is identical to the DP-UMPU test. The Bonferroni test can also be shown to be unbiased in this setting, however it still suffers a loss in power since it is not UMP. As in Figure \ref{fig:TwoSideTest_1.pdf}, the one-sided tests give upper bounds on the power.

In Figures \ref{fig:TwoSideTest_1.pdf}, \ref{fig:TwoSideTest_12.pdf}, and \ref{fig:TwoSideTest_5.pdf}, we see that all of the proposed tests have power equal to $\alpha=.05$ when the true value of $\theta$ is equal to the null. This confirms that all of our tests have type I error exactly $\alpha$, as claimed.

Figure \ref{fig:TwoSideN.pdf} compares the power of the tests as the sample size increases. In this simulation, we are testing $H_0: \ta_0 = .8$ versus $H_1: \ta_0\neq .8$ where the true value is $\ta = .75$. We use the values $\epsilon=.1$, $\delta=0$, and $\alpha=.05$. In this plot, we see again that the DP-UMP test for $H_0: \ta\leq \ta_0$ has more power than any of the other tests. The power of the UMPU and the approximate UMPU are indistinguishable, and the power of the Bonferroni test is slightly lower than either the UMPU or approximate UMPU tests. As we expect, the power of the DP-UMP test for $H_0: \ta\geq \ta_0$ goes to zero as $n\rightarrow \infty$. 

\begin{figure}
\begin{minipage}{.48\linewidth}
  \includegraphics[width=\linewidth]{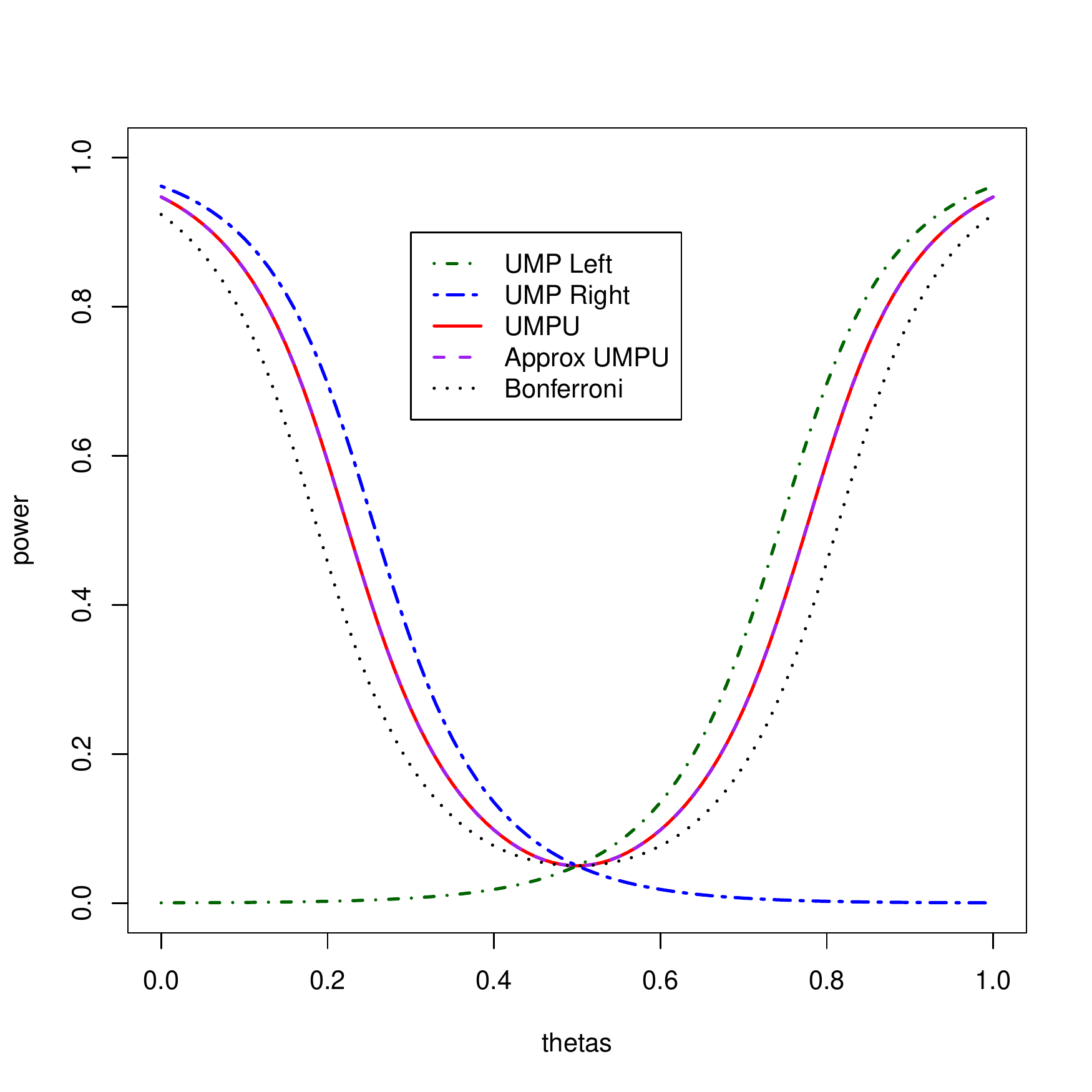}
      \captionof{figure}{$n=100$, $\ta_0=.5$, $\ep=.1$, $\de=0$, $\alpha=.05$. Testing $H_0: \ta=\ta_0$ versus $H_1:\ta_0\neq \ta_0$. $x$-axis is the truth.}
\label{fig:TwoSideTest_5.pdf}
\end{minipage}
\hspace{.02\linewidth}
\begin{minipage}{.48\linewidth}
  \includegraphics[width=\linewidth]{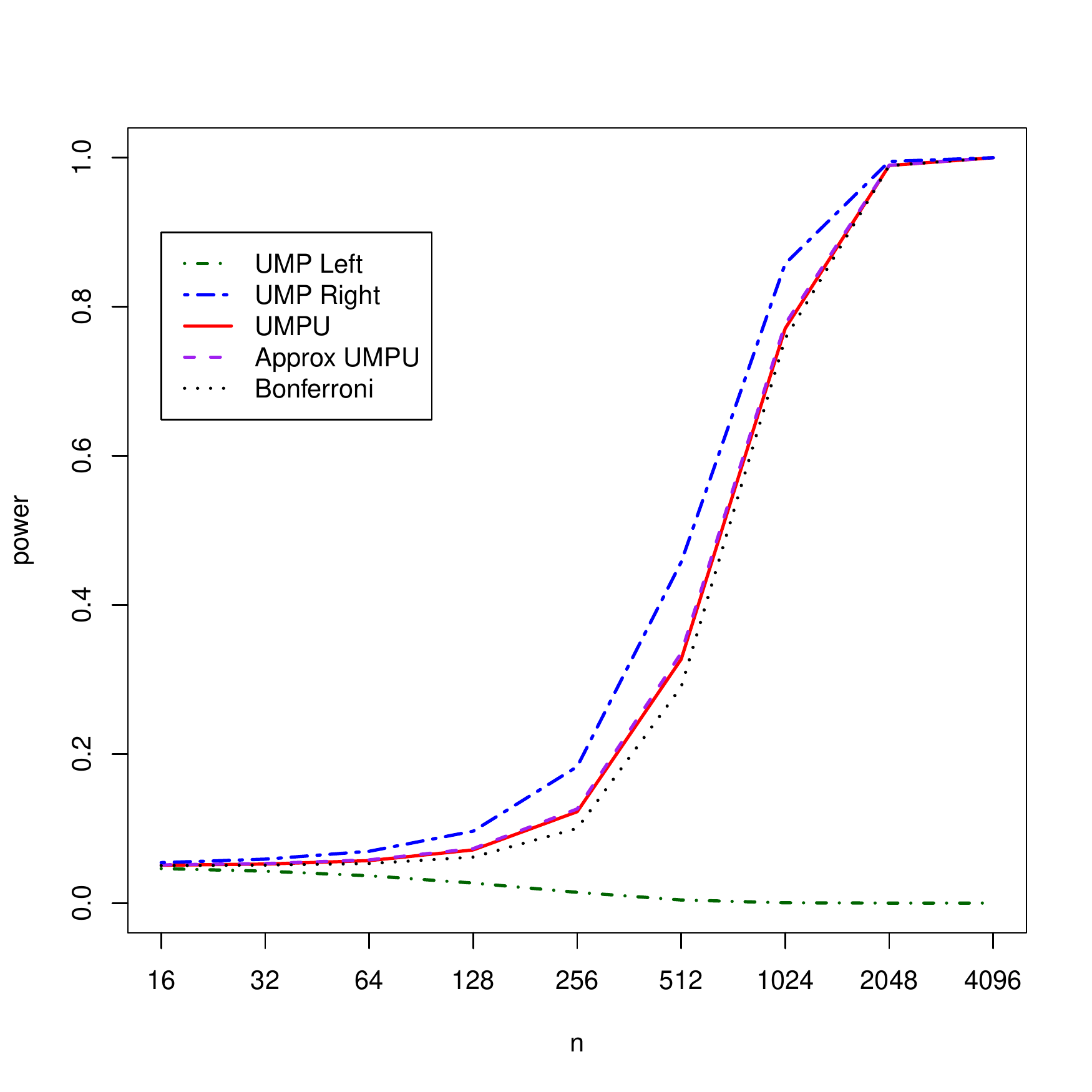}
    \caption{ $\ta_0 = .8$, truth= $.75$, $\ep=.1$, $\de=0$, $\alpha=.05$. Testing $H_0: \ta=\ta_0$ versus $H_1:\ta_0\neq \ta_0$. $x$-axis is the sample size $n$.}
      \label{fig:TwoSideN.pdf}
\end{minipage}
\end{figure}

\subsection{Two-sided confidence interval simulations}\label{s:Simulations3} 

In this section, we study the performance of the private confidence intervals given in Proposition \ref{prop:CI}. The label ``Approx UMPU'' corresponds to the interval $C^2_\al$ defined in Proposition \ref{prop:CI}, and ``Bonferroni'' corresponds to the interval $C^1_\al$ defined in Proposition \ref{prop:CI}.

In Figure \ref{fig:CI_30_1.pdf}, we compute the average width of the intervals depending on the true value of $\ta$ over 1000 replicates for each value of $\ta$. For this simulation, $n=30$, $\ep=1$, $\de=0$, and $\alpha=.05$. In Figure \ref{fig:CI_30_1.pdf}, we see that the approximately unbiased confidence interval achieves smaller width than the Bonferroni confidence interval for moderate $\ta$s, at the expense of larger widths for more extreme $\ta$s. At $\ta = \frac 12$, the approximately unbiased confidence interval is $97.8\%$ the width of the Bonferroni confidence interval, but at $\ta$ close to $0$ or $1$, the approximately unbiased confidence interval is $4.1\%$ wider than the Bonferroni confidence interval. 
The empirical coverage varied between $.93$ and $.962$ for both confidence intervals, with an average coverage of $0.9496$. The Monte Carlo standard errors for these estimates is $\sqrt{.95*(1-.95)/(1000)} = 0.0069$, suggesting that the confidence intervals have coverage $(1-\alpha) = .95$ as claimed  by Proposition \ref{prop:CI}.

In Figure \ref{fig:CI_N.pdf}, we explore the average width, but vary the sample size instead of $\theta$. For this simulation, the true value of $\theta$ is $\frac 12$, $\ep=1$, $\de=0$, and $\alpha=.05$. The average width is computed at each value of $n$, over 1000 replicates. We see that the approximately unbiased confidence interval consistently achieves a slightly smaller average width than the Bonferroni confidence interval. The largest discrepancy is at $n=16$, where is approximately unbiased confidence interval has an average width $97.53\%$ of the Bonferroni confidence interval. We see that as the sample size increases, the average width of these two confidence intervals becomes more similar.

\begin{figure}
\begin{minipage}{.48\linewidth}
    \includegraphics[width=\linewidth]{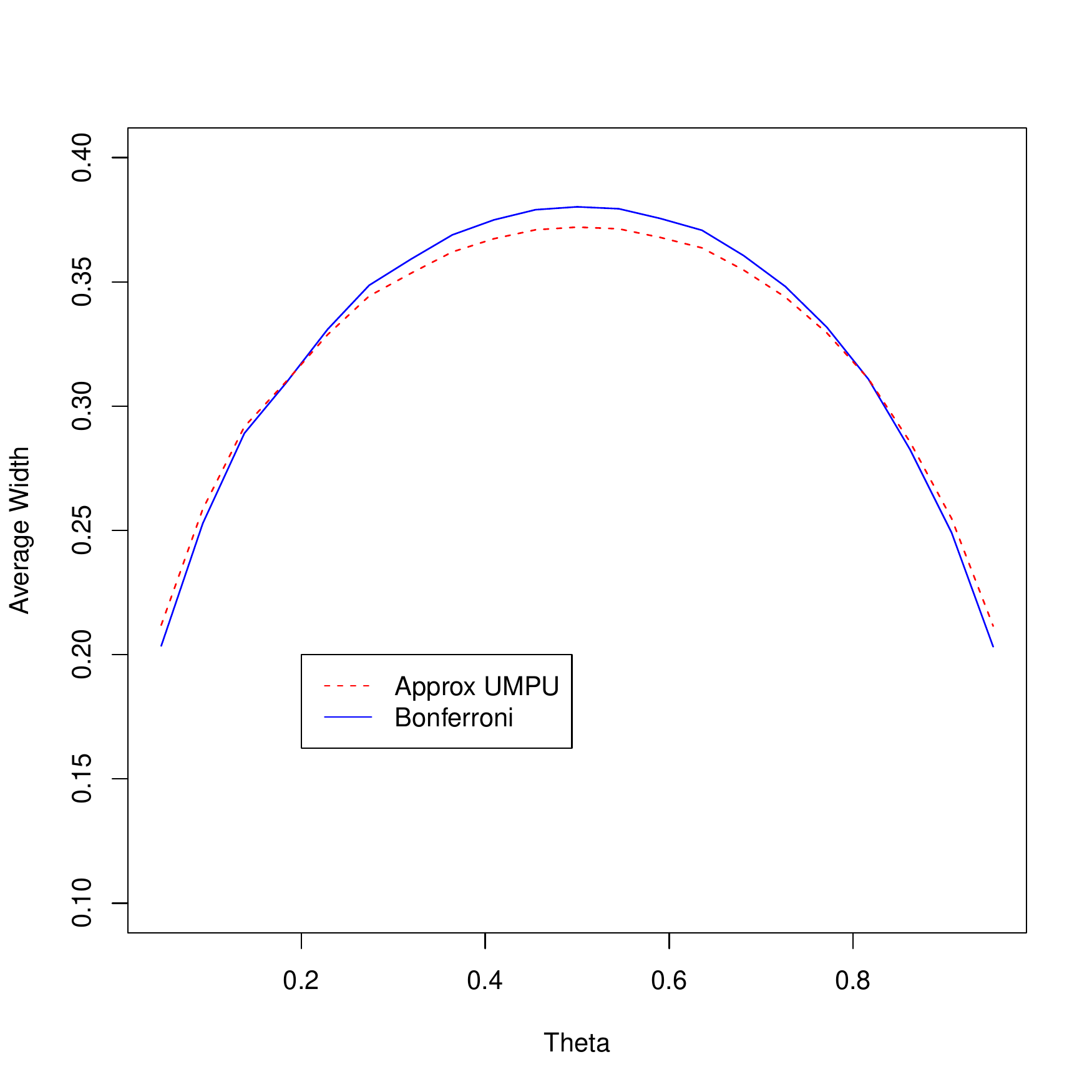}
    \captionof{figure}{$n=30$, $\ep=1$, $\de=0$, $\alpha=.05$. Average width of DP confidence intervals, over 1000 replications. $x$-axis is the true value of $\theta$.}
  \label{fig:CI_30_1.pdf}
\end{minipage}
\hspace{.02\linewidth}
\begin{minipage}{.48\linewidth}
   \includegraphics[width=\linewidth]{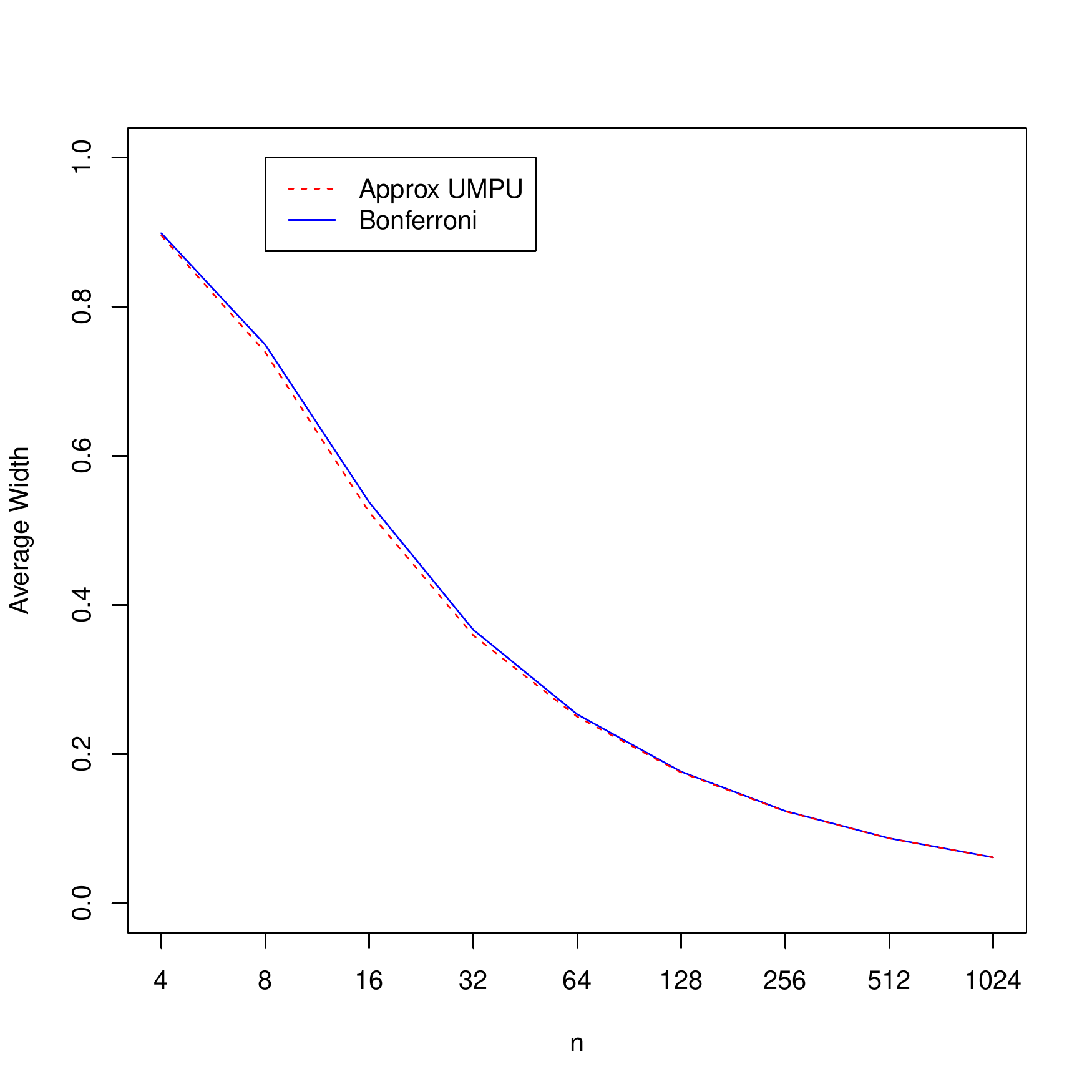}
    \captionof{figure}{$\ta=1/2$, $\ep=1$, $\de=0$, $\alpha=.05$. Average width of DP confidence intervals, over 1000 replications. $x$-axis is the sample size $n$.}
  \label{fig:CI_N.pdf}
\end{minipage}
\end{figure}
 \section{Discussion and future directions}\label{s:Discussion}
 In this paper, we  derived uniformly most powerful simple and one-sided tests for Bernoulli data among all DP $\alpha$-level tests. Previously, while various hypothesis tests under DP have been proposed, none have satisfied such an optimality criterion. While our initial DP-UMP tests only output `Reject' or `Fail to Reject', we showed that they can be achieved by post-processing a noisy sufficient statistic. This allows us to produce private $p$-values which agree with the DP-UMP tests. We also applied our techniques to produce two-sided tests, confidence intervals, and confidence distributions.

   The ability to produce private $p$-values and confidence intervals, rather than simply an accept/reject decision, has practical importance as well, since both the statistics and scientific community have been strongly arguing for providing more complete information on basic statistical inference when determining statistical significance, as the latter cannot and should not be equated with scientific significance \citep{Nuzzo2014,Wasserstein2016}.

 A simple, yet fundamental observation that underlies our results is that DP tests can be written in terms of linear constraints. This idea alone allows for a new perspective on DP hypothesis testing, which is particularly applicable to other discrete problems, such as multinomial models or difference of population proportions. Stating the problem in this form allows for the consideration of all possible DP tests, and allows the exploration of UMP tests through numerical linear program solvers.


We showed that for exchangeable data, DP tests need only depend on the empirical distribution. For binary data, the empirical distribution is equivalent to the sample sum, which is a complete sufficient statistic for the binomial model. However, in general it is not clear whether optimal DP tests are always a function of complete sufficient statistics as is the case for classical UMP tests. It would be worth investigating whether there is a notion of sufficiency which applies for DP tests.

 When $\de=0$, our optimal noise adding mechanism, the proposed Tulap distribution, is related to the discrete Laplace distribution, which \citet{Ghosh2009} and \citet{Geng2016} also found is optimal for a general class of loss functions. For $\de>0$, a truncated discrete Laplace distribution is optimal for our problem. Little previous work has looked into optimal noise adding mechanisms for $(\ep,\de)$-DP. \citet{Geng2016:Approximate} studied this problem to some extent, but did not explore truncated Laplace distributions. \citet{Steinke2018} proposes that truncated Laplace can be viewed as the canonical distribution for $(\ep,\de)$-DP in a way that Laplace is canonical for $(\ep,0)$-DP. Further exploration in the use of truncated Laplace distributions in the $(\ep,\de)$-DP setting may be of interest.

 In our work, we found that there was a close connection between our UMP tests and the staircase distribution, which \citet{Geng2016} show is universal utility maximizing for binary data. However, \citet{Brenner2014} showed that when the data are non-binary, there is no universal utility maximizing mechanism such as the staircase mechanism. As \citet{Canonne2018} discuss, this result seems to imply that in settings where the data is non-binary, it may not be possible to develop DP-UMP tests. 
 
  \section*{Acknowledgements}
  We would like to thank Vishesh Karwa and Matthew Reimherr for helpful discussions and feedback on previous drafts. 
  This work is supported in part by NSF Award No. SES-1534433 to The Pennsylvania State University. Part of this work was done while the second author was visiting the Simons Institute for the Theory of Computing.


\bibliographystyle{awan} 
 \bibliography{./DataPrivacyBib}{}

 \setcounter{section}{6}
\section{Detailed proofs and technical lemmas}\label{s:Appendix}
\begin{proof}[Proof of Theorem \ref{SuffStatThm}.]
  Define $\phi'$ by $\phi'(\ul x) = \frac{1}{n!} \sum_{\pi \in \sa(n)} \phi(\pi(\ul x))$, where $\sa(n)$ is the symmetric group on $n$ letters. First note that $\phi(\pi(\ul x))$ satisfies \eqref{DPBinary} for all $\pi \in \sa(n)$, and that $\int \phi(\pi(\ul x)) \ d\mu_\ta= \int\phi(\ul x)\ d\mu_\ta$. Then by exchangeability,
  \begin{align*}\int\phi'(\ul x)\ d\mu_\ta = \int \frac 1{n!} \sum_{\pi \in \sa(n)} \phi(\pi(\ul x))\ d\mu_\ta
    &= \frac 1{n!} \sum_{\pi \in \sa(n)} \int \phi(\pi(\ul x))\ d\mu_\ta\\
    &= \frac 1{n!} \sum_{\pi \in \sa(n)} \int\phi(\ul x)\ d\mu_\ta = \int \phi(\ul x)\ d\mu_\ta.
    \end{align*}
  To see that $\phi'$ satisfies $(\ep,\de)$-DP, we check condition \eqref{DPBinary}:
  \begin{align*}\phi'(\ul x) = \frac{1}{n!} \sum_{\pi \in \sa(n)} \phi(\pi (\ul x))
    &\leq \frac{1}{n!} \sum_{\pi \in \sa(n)} (e^\ep \phi(\pi(\ul x')) + \de)\\
    &= \frac{1}{n!} \sum_{\pi \in \sa(n)} e^\ep \phi(\pi(\ul x')) + \frac{1}{n!} \sum_{\pi \in \sa(n)} \de
    = e^{\ep}\phi'(\ul x') + \de
    \end{align*}
  \begin{align*}
    (1- \phi'(\ul x))&=\l(1- \frac{1}{n!} \sum_{\pi \in \sa(n)} \phi(\pi(\ul x))\r)
    = \frac{1}{n!} \sum_{\pi \in \sa(n)} (1- \phi(\pi(\ul x)))\\
    &\leq \frac{1}{n!} \sum_{\pi \in \sa(n)} \l(e^\ep (1- \phi(\pi(\ul x'))) + \de\r)
    =e^\ep \l(1- \phi'(\pi (\ul x'))\r) + \de.\qedhere
    \end{align*}
  \end{proof}
  
\begin{lem}\label{TulapLem}
  \begin{enumerate}[1)]
  \item Let $L\sim \mathrm{DLap}(b)$, $U\sim \mathrm{Unif}(-1/2,1/2)$, $G_1,G_2\iid \mathrm{Geom}(1-b)$, and $N_0\sim \mathrm{Tulap}(m,b,0)$, where the pmf of $L$ is $f_L(x) = \frac{1-b}{1+b}b^{|x|}$ for $x\in \ZZ$, and the pmf of $G_1$ is $f_{G_1}(x)=(1-p)^xp$ for $x\in \{0,1,2,\ldots\}$. Then $L+U+m\overset d= G_1-G_2+U+m\overset d= N_0$.
\item Let $N$ be the output of Algorithm \ref{SampleTulap} with inputs $m,b,q$. Then $N\sim \mathrm{Tulap}(m,b,q)$.
\item The random variable $N\sim \mathrm{Tulap}(m,b,q)$ is continuous and symmetric about $m$.
  \end{enumerate}
\end{lem}
\begin{proof}[Proof of Lemma \ref{TulapLem}.]
  \begin{enumerate}[1)]
\item  We know that $L\overset d= G_1-G_2$, as shown in \citet{Inusah2006}. Let $f_U(\cdot)$ denote the pdf of $U$, and $F_U$ denote the cdf of $U$. We will use the property that $f_U(x)=f_U(-x)$ and $F_U(-x) =1-F_U(x)$. Then the pdf of $L+U$ is 
  \begin{align*}
    f_{L+U}(x) &= f_U(x-[x])\l(\frac{1-b}{1+b}\r) b^{|[x]|}= \begin{cases}
f_U(x-[x]) \l(\frac{1-b}{1+b}\r)b^{-[x]}&[x]\leq 0\\
f_U(x-[x])\l(\frac{1-b}{1+b}\r) b^{[x]}&[x]>0.
\end{cases}
  \end{align*}
If $[x]\leq 0$, then we have 
\begin{align*}
  F_{L+U}(x) &= \int_{-\infty}^x f_U(t-[t]) \l( \frac{1-b}{1+b}\r) b^{-[t]} \ dt\\
&= \int_{-\infty}^{[x]-1/2} f_U(t-[t])\l( \frac{1-b}{1+b}\r) b^{-[t]} \ dt + \int_{[x]-1/2}^xf_U(t-[x])\l( \frac{1-b}{1+b}\r) b^{-[x]} \ dt\\
&= \sum_{t=-\infty}^{[x]-1}\l(\frac{1-b}{1+b}\r) b^{-t} +  \int_{[x]-1/2}^xf_U(t-[x])\l( \frac{1-b}{1+b}\r) b^{-[x]} \ dt\\
&= \frac{b^{-[x]+1}}{1+b}+F_U(x-[x])\l(\frac{1-b}{1+b}\r) b^{-[x]}\\&= \frac{b^{-[x]}}{1+b}(b+F_U(x-[x])(1-b)).
\end{align*}
Since, $L+U$ is symmetric about zero, as both $L$ and $U$ are symmetric about zero, for $[x]\geq 0$ we have $F_{L+U}(x) = 1-F_{L+U}(-x)$. The rest follows by replacing $x$ with $x-m$, and $F_U(x) = x+1/2$.
\item If $q=0$, then by part 1), it is clear that the output of Algorithm \ref{SampleTulap} has the correct distribution. If $q>0$, then by rejection sampling, we have that $N\sim \mathrm{Tulap}(m,b,q)$. For an introduction to rejection sampling, see \citet[Chapter 11]{Bishop2006}.
\item This property follows immediately from 1), and that $\mathrm{Tulap}(m,b,q)$ is truncated equally on both sides of $m$.\qedhere
\end{enumerate}
\end{proof}

 \begin{lem}\label{CLem}
Let $N\sim \mathrm{Tulap}(m,b,q)$ and let $t\in \ZZ$. Then 
$\ds F_{N}(t)= \begin{cases}
b^{-t}C(m)&t\leq [m]\\
1-b^tC(-m)&t>[m],
\end{cases}$
where $C(m) = (1+b)^{-1} b^{[m]} (b+([m]-m+1/2)(1-b))$.
 $C(m)$ is positive, monotone decreasing, and continuous in $m$. Furthermore, $b^{-[m]}C(m) = 1-b^{[m]}C(-m)$.
\end{lem}
\begin{proof}[Proof of Lemma \ref{CLem}.]
The form of the cdf at integer values is easily verified from Lemma \ref{TulapLem}. 
It is clear that $C(m)$ is positive. It is also clear that $C(m)$ is continuous and monotone decreasing for all $m\in \RR\setminus \{z+1/2\mid z\in \ZZ\}$. So, we will check that $C$ is continuous at $m=z+1/2$ for $z\in \ZZ$:
\[\lim_{\ep \downarrow 0} (1+b) C(z+1/2+\ep) = \lim_{\ep \downarrow 0} b^{z+1}(b+(1-\ep)(1-b)) = b^{z+1}\]
\[\lim_{\ep \downarrow 0} (1+b)C(z+1/2-\ep) = \lim_{\ep \downarrow 0} b^{z}(b+\ep(1-b)) = b^{z+1}.\]
Since $C$ is continuous on $\RR$ and monotone decreasing almost everywhere, it follows that $C$ is monotone decreasing on $\RR$ as well.

Call $\al(m) = [m]-m+1/2$, which lies in $[0,1]$. Note that $\al(-m) = -[m]+m+1/2=1-\al(m)$. Then 
\begin{align*}
  (1+b)b^{-[m]}C(m)&=b+\al(m)(1-b)
=b+(1-\al(-m)(1-b))
=b+(1-b)-\al(-m)(1-b)\\
&=(1+b)-(b+\al(-m)(1-b))
=(1+b)(1-b^{[m]}C(-m)).\qedhere
\end{align*}
\end{proof}

\begin{proof}[Proof of Lemma \ref{RecurrenceLem1}.]
First we show that 1) and 2) are equivalent. Clearly the $m$ is the same for both. We must show that for $p\in (0,1)$, $e^\ep p\leq 1-e^{-\ep}(1-p)$ whenever $p\leq \frac{1}{1+e^{\ep}}$, and $e^{\ep}p>1-e^{-\ep}(1-p)$ when $p>\frac{1}{1+e^\ep}$. Setting equal $e^\ep p = 1-e^{-\ep}(1-p)$ we find that $p =\frac{1}{1+e^\ep}$. As $p\rightarrow 1$, we have that $e^\ep p > 1- e^{-\ep}(1-p)$ and as $p\rightarrow 0$, we have $e^\ep p <1-e^{-\ep}(1-p)$. We conclude that 1) and 2) are equivalent.

Next we show that 2) and 3) are equivalent. First we show that $F_{N_0}(x-m)$ satisfies the recurrence relation in 2). Set $b=e^{-\ep}$. First we show that for $t\in \ZZ$ such that $t\leq [m]-1$, $F_{N_0}(t-m)\leq \frac{1}{1+e^\ep}$ and for $t\geq [m]$, $F_{N_0}(t-m)\geq \frac{1}{1+e^\ep}$. Since, $F_{N_0}(\cdot-m)$ is increasing, it suffices to check $t=[m]-1$ and $t=[m]$:
\[F_{N_0}([m]-1-m) = b^{-[m]+1+[m]} \frac{(b+([m]-m+1/2)(1-b))}{1+b}\leq \frac{b}{1+b}=\frac{1}{1+e^\ep}\]
\[F_{N_0}([m]-m) = b^{-[m]+[m]}\frac{(b+([m]-m+1/2)(1-b))}{1+b} \geq \frac{b}{1+b} = \frac{1}{1+e^\ep},\]
where we use the fact that $0\leq [m]-m+1/2\leq 1$. 
Now, let $t\in \ZZ$ and check three cases:
\begin{itemize}
\item Let $t< [m]$, then $e^\ep F_{N_0}(t-m)=e^\ep b^{-t}C(m)=b^{-(t+1)}C(m) = F_{N_0}(t+1-m)$.
\item Let $t=[m]$. Using Lemma \ref{CLem}, 
$1-e^{-\ep}(1-F_{N_0}(t-m)) 
= 1-b(1-b^{-[m]}C(m))
=1-b+b(1-b^{[m]}C(-m))
=1-b+b-b^{[m+1]}C(-m)
=F_{N_0}(t+1-m)$.
\item Let $t>m$. Then 
$1-e^{-\ep}(1-F_{N_0}(t-m))
=1-b(b^{t}C(-m))
=1-b^{t+1}C(-m)
=F_{N_0}(t+1-m)$.
\end{itemize}
Finally, for any value $c\in (0,1)$, we can find $m$ such that $F_{N_0}(0-m)=c$, by the intermediate value theorem. On the other hand, given $m$, set $\phi(0) = F_{N_0}(0-m)$.
\end{proof}

\begin{proof}[Proof of Lemma \ref{BingLem}.]
  Note that $(\phi_1 - \phi_2)(g-kf)\geq 0$ for almost all $x\in \mscr X$ (with respect to $\mu$). Then $\int (\phi_1 - \phi_2)(g-kf) \ d\mu \geq 0$. Hence, $\int \phi_1 g \ d\mu - \int \phi_2 g \ d\mu \geq k(\int \phi_1 f\ d\mu - \int \phi_2 f \ d\mu) \geq 0$.
\end{proof}

\begin{proof}[Proof of Theorem \ref{UMP1}.]
 First note that $\phi^*\in \mscr D_{\ep,0}^n$, since by Lemma \ref{RecurrenceLem1}, $\phi^*(x) =\min\{e^\ep \phi^*(x-1),1-e^{-\ep}(1-\phi^*(x-1))\}$. So, $\phi^*$ satisfies \eqref{DP1}-\eqref{DP4}. Next, since by Lemma \ref{CLem}, $F_{N_0}(x-m)$ is a continuous, decreasing function in $m$ with $\lim_{m\uparrow \infty} F_{N_0}(x-m)=0$ and $\lim_{m\downarrow -\infty} F_{N_0}(x-m)=1$, we can find $m$ such that $E_{\ta_0} \phi^*(x)=\al$ by the Intermediate Value Theorem. 

Now that we have argued that $\phi^*$ is a valid test, the rest of the result is an application of Lemma \ref{BingLem}. It remains to show that the assumptions are satisfied for the lemma to apply. Let $\phi\in \mscr D_{\ep,0}^n$ such that $E_{\ta_0} \phi(x)\leq \al$. 

We claim that either $\phi(x)=\phi^*(x)$ for all $x\in \{0,1,2,\ldots, n\}$ or there exists $y$ such that $\phi(y)<\phi^*(y)$. To the contrary, suppose that $\phi^*(x)\leq \phi(x)$ for all $x$ and there exists $z$ such that $\phi^*(z)<\phi(z)$. But this implies that $\phi^*(0)<\phi(0)$  (as we implied by the following paragraphs, by setting $y=0$). Then $E_{\ta_0} \phi^*(x)<E_{\ta_0} \phi(x)\leq \al$ since the pmf of $\mathrm{Binom}(n,\ta)$ is nonzero at $x=0$, contradicting the fact that $E_{\ta_0} \phi^*(x)=\al$. We conclude that there exists $y$ such that $\phi^*(y)>\phi(y)$. 

Let $y$ be the smallest point in $\{0,1,2,\ldots, n\}$ such that $\phi^*(y)\geq\phi(y)$. We claim that for all $x\geq y$, we have $\phi^*(x)\geq \phi(x)$. We already know that for $y=x$, the claim holds. For induction, suppose the claim holds for some $x\geq y$. By Lemma \ref{RecurrenceLem1}, we know that $\phi^*(x+1) = \min\{e^{\ep} \phi^*(x), 1-e^{-\ep} (1-\phi^*(x))\}$, and by constraints \eqref{DP1}-\eqref{DP4}, we know that $\phi(x+1) \leq \min\{e^{\ep} \phi(x), 1-e^{-\ep}(1-\phi(x))\}$. 
\begin{itemize}
\item Case 1: If $\phi^*(x)\leq \frac{1}{1+e^{\ep}}$, then by Lemma \ref{RecurrenceLem1},
$\phi^*(x+1) = e^\ep \phi^*(x)\geq e^\ep \phi(x)\geq \phi(x+1).$
\item Case 2: If $\phi^*(x)> \frac{1}{1+e^\ep}$, then by Lemma \ref{RecurrenceLem1},
$\phi^*(x+1) = 1-e^{-\ep}(1-\phi^*(x)) \geq 1-e^{-\ep}(1-\phi(x))\geq \phi(x+1).$
\end{itemize}
 We conclude that $\phi^*(x+1)\geq \phi(x+1)$. By induction, the claim holds for all $x\in \{y,y+1,y+2,\ldots, n\}$. So, we have that $\phi^*(x)\geq \phi(x)$ for $x\in \{y,y+1,y+2,\ldots, n\}$ and $\phi^*(x)\leq \phi(x)$ for $x\in \{0,1,2,\ldots, y-1\}$. Since $\mathrm{Binom}(n,\ta)$  has a monotone likelihood ratio in $\ta$, by Lemma \ref{BingLem} we have that $E_{\ta_1} \phi^*(x) \geq E_{\ta_1}\phi(x)$. We conclude that $\phi^*$ is UMP-$\al$ among $\mscr D_{\ep,0}^n$ for the stated hypothesis test.
\end{proof}

\begin{proof}[Proof of Lemma \ref{RecurrenceLem2}.]
  We will abbreviate $F(x)\defeq F_{N_0}(x-m)$, where $N_0 \sim \mathrm{Tulap}(0,b=e^{-\ep},0)$ to simplify notation. First we will show that 1) and 2) are equivalent. It is clear that $y$ and $m$ are the same in both. Next consider 
  $1-e^{-\ep}(1-p) + e^{-\ep}\de = e^\ep p+\de$, solving for $p$ gives $p=\frac{1-\de}{1+e^{\ep}}$. Considering as $p\rightarrow 0$ and $p\rightarrow 1$, we see that $1-e^{-\ep}(1-p)+e^{-\ep} \de \geq e^\ep p + \de$ when $p \leq \frac{1-\de}{1+e^{\ep}}$ and $1-e^{-\ep}(1-p)+e^{-\ep} \de \leq e^\ep p + \de$ when $p \geq \frac{1-\de}{1+e^{\ep}}$. 

Next solving $1-e^{-\ep}(1-p)+e^{-\ep}\de=1$ for $p$ gives $p=1-\de$. So, $1-e^{-\ep}(1-p)+e^{-\ep}\de\leq 1$ when $p\leq 1-\de$ and $1-e^{-\ep}(1-p)+e^{-\ep}\de\geq 1$ when $p\geq 1-\de$. Lastly, solving $e^{\ep}p+\de=1$ for $p$ gives $p = \frac{1-\de}{e^{\ep}}\geq \frac{1-\de}{1+e^{\ep}}$. Combining all of these comparisons, we see that 1) is equivalent to 2).

Before we justify the equivalence of 2) and 3), we argue the following claim. Let $\phi(x)$ be defined as in 3). Then $\phi(x)\leq\frac{1-\de}{1+e^{\ep}}$ if and only if $F(x)\leq\frac{1}{1+e^{\ep}}$. Suppose that $\phi(x) \leq \frac{1-\de}{1+e^{\ep}}$. Then $\frac{F(x)-q/2}{1-q} \leq \frac{1-\de}{1+e^{\ep}}$. Thus,
\begin{align*}
F(x)&\leq \frac{(1-q)(1-\de)}{1+e^{\ep}} + \frac q2\\
 &=\frac{1}{1+e^{\ep}}\l((1-q)(1-\de) + \l(\frac{b+1}{b}\r) \frac{q}{2}\r)\\
 &=\frac{1}{1+e^{\ep}} \l(\frac{(1-b)(1-\de)}{1-b+2\de b} + \l(\frac{b+1}{b}\r)\frac{\de b}{1-b+2\de b}\r)\\
 &=\frac{1}{1+e^{\ep}} (1-b+2 \de b)^{-1}((1-b)(1-\de) + (b+1)\de)\\
 &=\frac{1}{1+e^{\ep}}.
 \end{align*}
We are now ready to show that $\phi(x)$ as described in 3) fits the form of 2). 
\begin{itemize}
\item Suppose that $0<\phi(x)<\frac{1-\de}{1+e^{\ep}}$. By the above, we know that $F(x)\leq \frac{1}{1+e^{\ep}}$. By Lemma \ref{RecurrenceLem1},
  \begin{align*}
    e^\ep \phi(x) + \de = \frac{e^{\ep} F(x) - \frac{q}{2b}}{1-q} + \de
&=\frac{F(x+1) - \frac{q}{2}}{1-q} + \frac{\frac{q}{2}-\frac{q}{2b}}{1-q}+\de\\
&=\phi(x+1) + \frac{\de b}{1-b}\l(1-\frac{1}{b}\r) + \de
    =\phi(x+1).
    \end{align*}
\item Suppose that $\frac{1-\de}{1+e^{\ep}}<\phi(x) \leq 1-\de$. Then we have $F(x) >\frac{1}{1+e^{\ep}}$. Then 
  \begin{align*}
    1-e^{-\ep}(1-\phi(x)) + e^{-\ep} \de
&=1-e^{-\ep}\l(1-\frac{F(x) - q/2}{1-q}\r) + e^{-\ep} \de\\
&=(1-q)^{-1}\l(1-q - e^{-\ep} \l( 1-F(x) - q/2\r)\r) + e^{-\ep} \de\\
&= (1-q)^{-1}(1-e^{-\ep}(1-F(x)) + bq/2-q) + b\de\\
&= (1-q)^{-1}(F(x+1) - q/2) + \frac{(b-1)q/2}{1-q} + b\de\\
&= \phi(x+1) + \frac{\de b(b-1)}{1-b} + b\de\\
&= \phi(x+1).
  \end{align*}
\item Finally, we must show that if $\phi(x)=1$ then $\phi(x-1)\geq 1-\de$. It suffices to show that $F(x)\geq 1-q/2$ implies that $F(x-1) \geq (1-\de)(1-q)+q/2=1-(1/b)(q/2)$. We prove the contrapositive. Suppose that $F(x-1)<1-(1/b)(q/2)$. Then since $F$ satisfies property \eqref{DP3}, we know that
  \begin{align*}
    F(x) \leq 1-e^{-\ep}(1-F(x-1))&< 1-b(1-(1-(1/b)(q/2)))\\
    &= 1-b(1-1+(1/b)(q/2)) = 1-q/2.\end{align*}
    \end{itemize}

We have justified that $\phi(x)$ in 3) satisfies the recurrence relation in 2). Given $\phi'$ of the form in 2), with first non-zero entry at $y$, by Lemma \ref{CLem} and Intermediate Value Theorem, we can find $m\in \RR$ such that $\phi(y)=\phi'(y)$. We conclude that 1), 2), and 3) are all equivalent.
\end{proof}

\begin{proof}[Proof of Corollary \ref{OneSide1}.]
  First we show that $\phi^*$ is UMP-$\al$ for $H_0: \ta\leq \ta_0$ versus $H_1: \ta>\ta_0$. Since $\phi^*(x)$ is increasing and $\mathrm{Binom}(n,\ta)$ has a monotone likelihood ratio in $\ta$, $E_\ta \phi^*\leq E_{\ta_0}\phi^*=\al$ for all $\ta\leq \ta_0$ (property of MLR). By Theorem \ref{UMP1}, we know that $\phi^*(x)$ is most powerful for any alternative $\ta_1>\ta_0$ versus the null $\ta_0$. So, $\phi^*$ is UMP-$\al$.

  Next we show that $\psi^*$ is UMP-$\al$ for $H_1: \ta\geq \ta_0$ versus $H_1: \ta<\ta_0$. First note that $\sup_{\ta\geq\ta_0} \EE_{\ta} \psi^*=\al$. Let $\psi$ be another test with $\sup_{\ta\geq \ta_0}\EE_{\ta} \psi\leq \al$. Let $\ta_1<\ta_0$, we will show that $\EE_{\ta_1} \psi^*\geq \EE_{\ta_1} \psi$. Define $\twid \psi^*(x) = \psi^*(n-x) = 1-F_{N_0}(n-x-m_2) =F_{N_0}(x+m_2-n)$ and $\twid \psi(x) =\psi(n-x)$.
  Then using the map $(x,\ta) \mapsto (n-x,1-\ta)$, we have that $\EE_{X\sim (1-\ta_0)} \twid \psi^*(X) = \EE_{X\sim (1-\ta_0)} \psi^*(n-X)=\EE_{Y\sim \ta_0}\psi^*(Y)=\al$. By a similar argument for $\psi$, we have that both $\twid \psi^*$ and $\twid \psi$ are level $\al$ for $H_0: \ta\leq 1-\ta_0$ versus $H_1: \ta>1-\ta_0$. Since $\EE_{(1-\ta_0)}\twid\psi^* = \al$, and $\twid \psi^*(x) =F_{N_0}(x-m')$, we have that $\twid \psi^*$ is UMP-$\al$ for $H_0: \ta\leq (1-\ta_0)$ versus $H_1: \ta>(1-\ta_0)$. Then for $\ta_1<\ta_0$, 
  \[\EE_{X\sim \ta_1} \psi^*(X) = \EE_{Y\sim (1-\ta_1)} \twid\psi^*(Y)
    \geq \EE_{Y\sim (1-\ta_1)} \twid \psi(Y) = \EE_{X\sim \ta_1} \psi(X).\]
  We conclude that $\psi^*$ is UMP-$\al$ for $H_1: \ta\geq \ta_0$ versus $H_1: \ta<\ta_0$.\qedhere
\end{proof}

\begin{lem}\label{CDFLem}
  Observe $\ul x\in \mscr X^n$. Let $T: \mscr X^n \rightarrow \RR$, and let $\{\mu_{\ul x}\mid \ul x\in \mscr X^n\}$ be a set of probability measures on $\RR$, dominated by Lebesgue measure. Suppose that $\mu_{\ul x}$ is parameterized by $T(\ul x)$ and $\mu_{\ul x}$ has MLR in $T(\ul x)$. Then $\{\mu_{\ul x}\}$ satisfies $(\ep,\de)$-DP if and only if for all $H(\ul x_1,\ul x_2)=1$ and all $t\in \RR$,
  \begin{align}
    \mu_{\ul x_1}((-\infty,t)) &\leq e^{\ep} \mu_{\ul x_2}((-\infty,t))+\de,\\
\mu_{\ul x_1}((t,\infty))&\leq e^{\ep} \mu_{\ul x_2}((t,\infty)) + \de.
  \end{align}
\end{lem}


\begin{proof}[Proof of Lemma \ref{CDFLem}.]
  Let $\al \in [0,1]$ be given. We will only consider $B\subset \RR$ (Lebesgue measurable) such that $\mu_{\ul x_2}(B)=\al$. Then demonstrating $(\ep,\de)$-DP requires $\sup\limits_{\{B\mid \mu_{\ul x_2}(B)=\al\}} \mu_{\ul x_1}(B) \leq e^{\ep} \al + \de$. We interpret this problem as testing the hypothesis $H_0: \ul x=\ul x_2$ versus $H_1: \ul x=\ul x_1$, using the rejection region $B$, where $\al$ is the type I error, and $\mu_{\ul x_1}(B)$ is the power. We know that $\sup\limits_{\{B\mid \mu_{\ul x_2}(B)=\al\}} \mu_{\ul x_1}(B)$ is achieved by the Neyman-Pearson Lemma. Since $\mu_{\ul x}$ has an MLR in $T(\ul x)$, $\arg\sup_{\{B\mid \mu_{\ul x_2}(B)=\al\}} \mu_{\ul x_1}(B)$ is either of the form $(-\infty,t)$ or $(t,\infty)$, depending on whether $T(\ul x_1)$ is greater or lesser than $T(\ul x_2)$. Since $\mu_{\ul x_1}$ is dominated by Lebesgue measure for all $\ul x_1$, $\mu_{\ul x_2}((-\infty,t))$ is continuous in $t$, which allows us to achieve exactly $\al$ type I error.
\end{proof}

\begin{proof}[Proof of Theorem \ref{TulapDP}.]
Let $Z \sim \mathrm{Tulap}\l(T(x),b=e^{-\ep},\frac{2\de b}{1-b+2\de b}\r)$. We know that the distribution of $Z$ is symmetric with location $T(x)$, and the pdf $f_Z(t)$ is increasing as a function of $|t-T(x)|$. It follows that $f_Z(t)$ has a MLR in $T(x)$. By Lemma \ref{RecurrenceLem2}, we know that $\phi(x)=F_Z(m)$ satisfies \eqref{DP1}-\eqref{DP4}, so by Lemma \ref{CDFLem}, we have the desired result. 
\end{proof}


\begin{proof}[Proof of Theorem \ref{PValueResult}.]
  We denote by $F_{Z\sim \ta_0}(\cdot)$ the cdf of the random variable $Z$, distributed as $Z\mid X\sim \mathrm{Tulap}(X,b,q)$ and $X\sim \mathrm{Binom}(n,\ta_0)$.
  \begin{enumerate}
  \item First we show that $p(\ta_0,Z)$ is a $p$-value, according to Definition \ref{def:pvalue}. 
    To this end, consider
    
    \begin{align*}
      \sup_{\ta\leq \ta_0} P_{\substack{Z| X \sim \mathrm{Tulap}(X,b,q)\\X\sim \mathrm{Binom}(n,\ta)}}(p(\ta,Z) \leq \al)
&= P_{\substack{Z| X \sim \mathrm{Tulap}(X,b,q)\\X\sim \mathrm{Binom}(n,\ta_0)}}(p(\ta_0,Z)\leq \al)\\
    \end{align*}
    using the fact that $X$ has a monotone likelihood ratio in $\ta$. Note that $p(\ta_0,Z) =1-F_{Z\sim \ta_0}(Z)$.
    When $X\sim \mathrm{Binom}(n,\ta_0)$, we have that $p(\ta_0,Z) = 1-F_{Z\sim \ta_0}(Z)\sim \mathrm{Unif}(0,1)$. So,
\[P_{\substack{Z| X \sim \mathrm{Tulap}(X,b,q)\\X\sim \mathrm{Binom}(n,\ta_0)}}(p(\ta_0,Z)\leq \al)
  =P_{U\sim \mathrm{Unif}(0,1)}( U\leq \al)=\al.\]
\item Let $N\sim \mathrm{Tulap}(0,b,q)$, and recall from Theorem \ref{UMP2} that the UMP-$\al$ test for $H_0: \ta\leq \ta_0$ versus $H_1: \ta>\ta_0$ is $\phi^*(x) = F_N(x-m)$, where $m$ satisfies $E_{\ta_0}\phi^*(x) =\al$. We can write $\phi^*$ as 
  \begin{align*}
    \phi^*(x) = F_N(x-m)
&= P_{N\sim \mathrm{Tulap}(0,b,q)}(N\leq X-m\mid X)\\
&= P_N(X+N\geq m\mid X)
= P_{Z| X \sim \mathrm{Tulap}(X,b,q)} (Z\geq m\mid X)
  \end{align*}
where $m$ is chosen such that 
\begin{align*}
\al = E_{X\sim \ta_0} \phi^*(X)
 &= E_{X\sim \ta_0} P_{Z| X \sim \mathrm{Tulap}(X,b,q)}(Z\geq m \mid X) \\
&= P_{\substack{Z| X\sim \mathrm{Tulap}(X,b,q)\\X\sim \mathrm{Binom}(n,\ta_0)}} (Z\geq m)
=1-F_{Z\sim \ta_0}(m),
\end{align*}
where $F$ is the cdf of the marginal distribution of $Z$, where $Z| X \sim \mathrm{Tulap}(X,b,q)$ and $X\sim \mathrm{Binom}(n,\ta_0)$. From this equation, we have that $m$ is the $(1-\al)$-quantile of the marginal distribution of $Z$. 

Let $R| X \sim \mathrm{Bern}(\phi^*(X))$ and $Z|X\sim \mathrm{Tulap}(X,b,q)$. Then 
\begin{align*}
  R| X\overset d = I(Z\geq m) \mid X&\overset d=I(F_{Z\sim \ta_0}(Z) \geq F_{Z\sim \ta_0}(m))\mid X \\
  &\overset d=I\l(1-\al \leq F_{Z\sim \ta_0} (Z)\r)  | {X}\\
&\overset d= I\l( p(\ta_0,Z)\leq \al\r)|X.\\
\end{align*}
Taking the conditional expected value $\EE(\cdot\mid X)$  of both sides gives
\[\phi^*(x) = E(R\mid X) = P_{Z| X\sim \mathrm{Tulap}(X,b,q)}(p(\ta_0,Z)\leq \al\mid X).\]
\item Let $p'(X)$ be any other $(\ep,\de)$-DP $p$-value, and let $\ta_1>\ta_0$. We wish to show that
  \[P_{X\sim \ta_1,N}(p(\ta_0,X+N)\leq \alpha)\geq P_{X\sim \ta_1}(p'(X)\leq \alpha)\]
  However, the left side is just $\EE_{X\sim \ta_1}\phi^*(X)$, the power of the DP-UMP test, and the right is the power of the corresponding DP test of $p'$. Since $\phi^*$ is uniformly most powerful among $(\ep,\de)$-DP tests, the inequality is justified.

\item We can express $p(\ta_0,Z)$ in the following way:
  \begin{align*}
    p(\ta_0,Z)&= P_{\substack{X\sim \mathrm{Binom}(n, \ta_0)\\ N\sim \mathrm{Tulap}(0,b,q)}} (X+N \geq Z)
= P_{X,N}(-N\leq X-Z)\\
&= E_{X\sim \mathrm{Binom}(n, \ta_0)} P_{N}(N\leq X-Z\mid X)
= E_{X\sim \mathrm{Binom}(n,\ta_0)}F_N(X-Z)\\
&= \sum_{x=0}^n F_N(x-Z)\binom nx \ta_0^x(1-\ta_0)^{n-x},
  \end{align*}
which is just the inner product of the vectors $\ul F$ and $\ul B$ in algorithm \ref{PValueAlgorithm}.\qedhere
  \end{enumerate}
\end{proof}

\begin{proof}[Proof of Theorem \ref{UMPU}.]
  We must show that there exists $k$ and $m$ which solve the two equations, and then argue that $\phi^*$ is UMP among all level $\al$ tests in $\mscr D_{\ep,\de}^n$. The proof is inspired by the Generalized Neyman Pearson Lemma \citet[Theorem 3.6.1]{Lehmann2008}, and has a similar strategy as Theorem \ref{OneSide1}.

  Let $\ta_1 \neq \ta_0$. We will show that $\phi*$ is most powerful among unbiased size $\al$ tests in $\mscr D_{\ep,\de}^n$ for testing $H_0: \ta=\ta_0$ versus $H_1: \ta=\ta_1$.  Set $f_1(x) = \binom nx \ta_0^x(1-\ta_0)^{n-x}$, $f_2(x) = (x-n\ta_0) \binom nx \ta_0^x(1-\ta_0)^{n-x}$, and $f_3(x) = \binom nx \ta_1^x(1-\ta_1)^{n-x}$. Let $\phi\in \mscr D_{\ep,\de}^n$ be any unbiased size $\alpha$ test, not identical to $\phi^*$.

  \begin{enumerate}
  \item There exists $k,m\in \RR$ such that $\phi^*$ satisfies  $\EE_{X\sim \ta_0} (X-n\ta_0) \phi(X)=0$ and $\EE_{X\sim \ta_0} \phi(X) = \al$.
    \begin{proof}
      Set $g_1(k,m) = \EE_{X\sim\ta_0}\phi^*(X) - \al$ and $g_2(k,m) = \EE_{X\sim \ta_0}(X-n\ta_0) \phi^*(X)$. We need to show that there exists $k$ and $m$ such that $g_1(k,m)=g_2(k,m)=0$.



      Then, $g_2$ partitions $\RR^2$ into two disjoint regions: the pairs $(k,m)$ such that $g_2\leq 0$ and those such that $g_2>0$. We claim that the solutions to $g_1=0$ form a curve. To see this, notice that for any $k\in \RR$, there exists a unique $m\in \RR$ such that $g_1(k,m)=0$. In particular, when $k=-1$, and $m(0)$ is the value such that $g_1(0,m(0))=0$, then $g_2(0,m(0))>0$ (check). Similarly, for $k=n+1$, $g_2(n,m(n))<0$.

      Since, both $g_1$ and $g_2$ are continuous functions, by the intermediate value theorem, there exists $-1\leq k\leq n+1$ and $m\in \RR$ such that $g_1(k,m)=g_2(k,m)=0$. 
    \end{proof}

  \item  Since $\phi$ is unbiased, it's power must have a local minimum at $\ta_0$ so, $\frac d {d\theta} \EE_\ta \phi\Big|_{\ta=\ta_0}=0$. This is equivalent to requiring that $\sum \phi f_2=0$.
    \begin{proof}
      We calculate the derivative of the power:
      \begin{align*}
        \frac{d}{d\theta} \beta_\phi&= \frac d {d\theta} \sum_{x=0}^n \binom nx \ta^{x}(1-\ta)^{n-x}\phi(x)\\
                                    &=\sum_{x=0}^n \binom nx \ta^x(1-\ta)^{n-x} \phi(x)\l( \frac{x}{\ta} + \frac{x-n}{1-\ta}\r)\\
                                    &=\frac{1}{\ta(1-\ta)} \EE_{\ta}(X-n\ta)\phi(X)\\
        &= \frac{1}{\ta(1-\ta)} \sum \phi f_2
      \end{align*}
    \end{proof}
    
  \item 
    There exists $y_l \leq k\leq y_u$ (integers) such that $\phi^*(x) \geq \phi(x)$ when $x\geq y_u$ or $x\leq y_l$, and $\phi^*(x)\leq \phi(x)$ when $y_l\leq x\leq y_u$.
    \begin{proof}
      Since $\phi$ is not identical to $\phi^*$, there exists $x\in\{0,\ldots, n\}$ such that $\phi^*(x) \neq \phi(x)$. 
      If $\phi^*(x)> \phi(x)$ for all $0\leq x\leq n$, then set $y_l=\floor(k)$ and $y_u=\ceil(k)$. If $\phi^*(x)\leq\phi(x)$ for all $0\leq x\leq n$, then it cannot be that $\phi$ is size $\al$. We conclude that there exists a value $y$ such that $\phi^*(y)> \phi(y)$. If $y>k$, then for every $x\geq y$, $\phi^*(x)> \phi(x)$, since $\phi^*$ increases as much as possible. Alternatively, if $y<k$ then for all $x\leq y$, $\phi^*(x)< \phi(x)$. So, we have that either $y_l\leq k$ exists or $y_u\geq k$ exists. We need to show that both exist.

      Suppose without loss of generality that $y_u\geq k$ exists, and suppose to the contrary that $y_l\leq k$ does not exist. Then it is the case that $\phi^*(x)> \phi(x)$ when $x\geq y_u$ and $\phi^*(x)\leq \phi(x)$ when $x<y_u$. Notice that $\frac{f_2(x)}{f_1(x)} = x-n\ta_0$. So, $\phi^*(x)\geq \phi(x)$ if and only if $f_2(x)\geq cf_1(x)$ for the constant $c=y_u-\frac 12 -n\theta_0$. Then for all $x=0,1,2,\ldots, n$,
      \[  (\phi^*(x) - \phi(x))(f_2(x) - cf_1(x))\geq 0,\]
      and $(\phi^*(y_u) - \phi(y_u))(f_2(y_u) - cf_1(y_u))>0$. 
      Summing over $x$ gives
      \begin{align*}
        \sum_{x=0}^n (\phi^*(x) - \phi(x))(f_2(x) - cf_1(x))&>0\\
        \sum \phi^*f_2 - \sum \phi f_2 - c\sum \phi^*f_1 + c\sum \phi f_1&>0\\
        -\sum \phi f_2&> 0\\
        \sum \phi f_2&<0,
      \end{align*}
      We see that $\phi$ is not unbiased, contradicting our initial assumption. We conclude that both $y_l$ and $y_u$ exist.
      
    \end{proof}
  \item There exists $k_1,k_2\in \RR$ such that $f_3(x)\geq k_1f_1(x) + k_2f_2(x)$ when $x\not\in (y_l,y_u)$ and $f_3(x)\leq k_1f_1(x) + k_2f_2(x)$ when $x\in (y_l,y_u)$.
    \begin{proof}
      We need to consider what forms the set $\{x\mid f_3(x)\geq k_1f_1(x)+k_2f_2(x)\}$ can take on. These are solutions to
      \begin{align*}
        1&\geq \frac{k_1f_1(x)+k_2f_2(x)}{f_3(x)}\\
        1&\geq (k_1+k_2(x-n\ta_0)) \frac{(1-\ta_0)^n}{(1-\ta_1)^n} \frac{\ta_0^x(1-\ta_0)^x}{\ta_1^x(1-\ta_1)^x}\\
        \left(\frac{1-\ta_1}{1-\ta_0}\right)^n\left( \frac{\ta_1(1-\ta_1)}{\ta_0(1-\ta_0)}\right)^x
        &\geq k_1-k_2n\ta_0 + k_2x.
      \end{align*}
      The left side is either convex or constant. The right side is linear. If the left is strictly convex (when $\ta_1 \neq 1-\ta_0$), we can always choose $k_1$ and $k_2$ such that the set of solutions is of the form $(-\infty,y_l]\cup[y_u,\infty)$. If $\ta_1 = 1-\ta_0$, set $k_1 = \l(\frac {1-\ta_1}{1-\ta_0}\r)^n$ and $k_2=0$.
    \end{proof}
  \item The test $\phi^*$ is more powerful than $\phi$ at any $\ta_1$.
    \begin{proof}
      We have established that $\phi^*(x)\geq \phi(x)$ for $f_3(x)\geq k_1f_1(x)+k_2f_2(x)$ and $\phi^*(x)\leq \phi(x)$ for $f_3(x) \leq k_1f_1(x)+k_2f_2(x)$. Then
g      \[(\phi^*(x) - \phi(x))(f_3(x)-k_1f_1(x)-k_2f_2(x))\geq 0\]
      Then
      \begin{align*}
        \sum_{x=0}^n (\phi^*(x) - \phi(x))(f_3(x)-k_1f_1(x)-k_2f_2(x))&\geq0\\
        \sum_{x=0}^n \phi^*(x)f_3(x)-\sum_{x=0}^n \phi(x) f_3(x) &\geq 0\\
        \EE_{X\sim \ta_1}\phi^*(X)\geq \EE_{X\sim \ta_1}\phi(X).
      \end{align*}
      Since our argument does not depend on the choice of $\ta_1$, we conclude that $\phi^*$ is more powerful than any other size $\al$ unbiased test in $\mscr D_{\ep,\de}^n$. Finally, noting that by taking $\phi(x)\defeq \al$, we see that $\phi^*$ is indeed unbiased. Hence, $\phi^*$ is the UMP-$\al$ among unbiased tests in $\mscr D_{\ep,\de}^n$. 
    \end{proof}
  \end{enumerate}
\end{proof}

\begin{proof}[Proof of Corollary \ref{UMPUhalf}.]
  We have to show that when using $k=\frac n2$, $\phi$ is unbiased. Call $A = \EE_{X\sim 1/2} (X-\frac n2) \phi(X)$. Then
  \begin{align*}
    A&= \l(\frac 12 \r)^{n}\sum_{x=0}^n \binom nx \l(x-\frac n2\r) \phi(x)\\
     &= \l(\frac 12 \r)^n \sum_{y=0}^n \binom ny \l(\frac n2 -y\r) \phi(n-y)\\
     &= \l(\frac 12 \r)^n \sum_{y=0}^n \binom ny \l(\frac n2 - y\r) \phi(y)\\
    &=-A
  \end{align*}
  where we made the substitution $y=n-x$, and used the fact that both $\binom{n}{\cdot}$ and $\phi(\cdot)$ are symmetric about $k=\frac n2$. We see that $A = -A$, which implies that $A = 0$. 
\end{proof}

\begin{proof}[Proof of Proposition \ref{prop:Nearly}.]
  Call $p^*(\ta_0, Z)$ the output of Algorithm \ref{NearlyUnbiasedPvalue}. First we will understand the distribution of $p^*(\ta_0, Z)$ when $\ta = \ta_0$:
  \begin{align*}
    p^*(\ta_0, Z)
    &=p(\ta_0, T+n\ta_0) + 1-p(\ta_0, n\ta_0-T)\\
    &=P_{Z\sim \ta_0}(Z\geq T+n\ta_0) + P_{Z\sim \ta_0}(Z\leq n\ta_0-T)\\
    &=P_{Z\sim \ta_0} (Z-n\ta_0\geq T \text{ or } Z-n\ta_0\leq -T)\\
    &= P_{Z\sim \ta_0}(|Z-n\ta_0|\geq T)\\
    &=1-F_T(T)\\
    &\sim U(0,1).
  \end{align*}

  Since $p^*(\ta_0, Z) \sim U(0,1)$, we have that
  $P_{\ta_0}(p^*(\ta_0, Z)\geq \al) = \al.$ 
 The $p$-value satisfies $(\ep,\de)$-DP since it is a post-processing of $Z$.
\end{proof}

\begin{proof}[Proof of Proposition \ref{prop:Unbiased}.]
  In the full proof of Theorem \ref{UMPU}, we saw that if $\phi$ is of the form in Theorem  \ref{UMPU} and  $\EE_{\ta_0} (X-n\ta_0)\phi(X)=0$, then $\phi$ is unbiased. Let $\phi$ be the test in Proposition \ref{prop:Nearly}. Then it suffices to show that $\lim_{n\rightarrow \infty} \EE_{\ta_0}\frac{X-n\ta_0}{\sqrt {n\ta_0(1-\ta_0) }}\frac{\phi(X)}{\sqrt n} = 0$. We begin by recalling that if $X\sim \mathrm{Binom}(n,\ta_0)$ then by the Central Limit Theorem, we have that
  \[\frac{X-n\ta_0}{\sqrt{n\ta_0(1-\ta_0)}} \overset d\rightarrow N(0,1).\]
Using this we have
  \begin{align*}
    \lim_{n\rightarrow \infty} \EE_{\ta_0}\frac{X-n\ta_0}{\sqrt {n\ta_0(1-\ta_0) }}\frac{\phi(X)}{\sqrt n}
    &=\lim_{n\rightarrow \infty} \EE_{Z\sim N(0,1)} Z \frac{\phi(Z\sqrt{n\ta_0(1-\ta_0)} + n\ta_0)}{\sqrt n}\\
    &= \lim_{n\rightarrow \infty} \EE_{Z\sim N(0,1)} Z \phi'(Z) 
  \end{align*}
  where
  \[\phi'(z)= \frac{1}{\sqrt n} \phi(z\sqrt{n\ta_0(1-\ta_0)} + n\ta_0)
    =\frac{1}{\sqrt n } \begin{cases}
      F_N(z\sqrt {n\ta_0(1-\ta_0)} - m)&\text{if } z\geq 0\\
      F_N(-z\sqrt {n \ta_0(1-\ta_0)} - m)&\text{if } z<0
    \end{cases}\]
  (We are assuming that $k = n\ta_0$) Notice that $\phi'$ is symmetric about $0$. So,
  \begin{align*}
    \EE_{Z\sim N(0,1)} Z\phi'(Z)
    &= \EE(Z\phi'(Z)I(Z\leq 0) + \EE(Z\phi'(Z)I(Z\geq 0))\\
                               &=-\EE(Z\phi'(Z)I(\geq 0) + \EE(Z\phi'(Z)I(Z\geq 0))\\
                               &=0
  \end{align*}
\end{proof}

\begin{proof}[Proof of Proposition \ref{UMAU}.]
  It is easy to verify that $C^*$ is unbiased, and has the appropriate coverage. Suppose to the contrary that $C^*$ is not UMA among DP unbiased confidence intervals. Then there exists two values $\ta\neq \ta_0$ and another DP confidence interval $C'$, which is unbiased with coverage $(1-\alpha)$ such that
  \[P_\ta(\ta_0 \in C')<P_{\ta}(\ta_0\in C^*)]\]
  or equivalently,
  \begin{equation}\label{FalseCoverage}
    P_\ta(\ta_0 \not \in C')>P_\ta(\ta_0 \not \in C^*).
  \end{equation}
  At this point, note that $\phi'(x) = P(\ta_0\not \in C'\mid X)$ is a size $\alpha$, unbiased DP test for $H_0: \ta=\ta_0 $ versus $H_1: \ta\neq \ta_0$. Furthermore, $\phi^*(x) = P(\ta_0 \not \in C^*\mid X)$ is the size $\alpha$ DP-UMPU test from Theorem \ref{UMPU}. But then equation \eqref{FalseCoverage} is equivalent to $\EE_\ta \phi'>\EE_\ta\phi^*$, which implies that $\phi^*$ is not the DP-UMPU. 
\end{proof}

  \begin{proof}[Proof of Corollary \ref{BonferroniCI}.]
    Suppose to the contrary that there exist $\ta_0\neq \ta_1$ and a $(\ep,\de)$-DP confidence interval $C'$ with coverage $1-\alpha/2$ such that
    \[P_{\ta_1}(\ta_0 \in C')<P_{\ta_1}(\ta_0 \in C^1_\alpha),\]
    or equivalently,
    \begin{equation}\label{falseCoverage1}
      P_{\ta_1}(\ta_0 \not \in C')>P_{\ta_1}(\ta_0 \not \in C^1_\alpha).
      \end{equation}
    We can then construct two hypothesis tests $\phi^1(x) = P(\ta_0 \not \in C^1_\alpha(x)\mid x)$, and $\phi'(x) = P(\ta_0 \not \in C'(x)\mid x)$ for $H_0: \ta = \ta_0$ versus $H_1: \ta \neq \ta_0$. Note that $\phi^1(x)$ is the test from Proposition \ref{prop:Bonferroni} at size $\alpha$, and $\phi'$ has size $\alpha/2$. Now,  \eqref{falseCoverage1} implies that $\EE_{\ta_1}\phi' >\EE_{\ta_1}\phi^1_\alpha$, which contradicts Proposition \ref{prop:Bonferroni}, which states that  $\phi'$ must be uniformly more powerful than $\phi'$. 
  \end{proof}
\end{document}